\documentclass[11pt]{article}

\usepackage{latexsym}
\usepackage{amssymb}
\usepackage{amsthm}
\usepackage{amscd}
\usepackage{amsmath}
\usepackage{mathrsfs}
\usepackage[all]{xy}
\usepackage{graphicx} 
\usepackage{hyperref} 
\usepackage[usenames,dvipsnames]{color}

\usepackage{color}

\theoremstyle{definition}
\newtheorem{theorem}{Theorem}[section]

\theoremstyle{definition}

\newtheorem{lemma}{Lemma}[section]
\theoremstyle{definition}

\theoremstyle{definition}
\newtheorem* {notation}{Notation}
\theoremstyle{definition}

\newtheorem{proposition}{Proposition}[section]
\newtheorem{corollary}{Corollary}[section]
\newtheorem* {remark}{Remark}
\newtheorem{example}{Example}[section]
\theoremstyle{definition}

\theoremstyle{definition}
\newtheorem* {remarks}{Remarks}

\xyoption{dvips}

\usepackage{fullpage}

\numberwithin{equation}{section}

\newcommand{\coadj}{ {\mathrm{coadj} } }
\newcommand{\coaux}{ {\mathrm{coaux} } }
\newcommand{\aux}{ {\mathrm{aux} } }
\newcommand{\adj}{ \mathrm{adj} }

\newcommand{\rml}{\mathrm R \setminus \mathrm L}
\newcommand{\larc}[1]{\hspace{-.4ex}\overset{#1}{\frown}\hspace{-.4ex}}

\def\({\left(}
\def\){\right)}
  \newcommand{\sP}{\mathscr{S}} \newcommand{\FF}{\mathbb{F}}  \newcommand{\CC}{\mathbb{C}}     \newcommand{\cP}{\mathcal{P}} \newcommand{\cA}{\mathcal{A}}
\newcommand{\cQ}{\mathcal{Q}} \newcommand{\cK}{\mathcal{K}}  \newcommand{\cC}{\mathcal{C}}
\newcommand{\cR}{\mathcal{R}}  \newcommand{\cS}{\mathcal{S}}\newcommand{\cT}{\mathcal{T}} 
 \def\NN{\mathbb{N}}
\def\tr{\mathrm{tr}}         \def\GL{\mathrm{GL}}   \def\Res{\mathrm{Res}}    \def\spanning{\textnormal{-span}}   
\def\Irr{\mathrm{Irr}}

   \newcommand{\fkn}{\mathfrak{n}}\newcommand{\supp}{\mathrm{supp}}

\newcommand{\One}{{1\hspace{-.14cm} 1}}

\newcommand{\fka}{\mathfrak{a}}

\def\fk{\mathfrak}

\def\barr{\begin{array}}
\def\earr{\end{array}}
\def\ba{\begin{aligned}}
\def\ea{\end{aligned}}
\def\be{\begin{equation}}
\def\ee{\end{equation}}
\def\cS{\mathcal{S}}

\makeatletter
\renewcommand{\@makefnmark}{\mbox{\textsuperscript{}}}
\makeatother

\allowdisplaybreaks[1]

\UseCrayolaColors

\begin{document}
\title{Superclasses and supercharacters of normal pattern subgroups of the unipotent upper triangular matrix group}
\author{Eric Marberg\footnote{This research was conducted with government support under
the Department of Defense, Air Force Office of Scientific Research, National Defense Science
and Engineering Graduate (NDSEG) Fellowship, 32 CFR 168a.} \\ Department of Mathematics \\ Massachusetts Institute of Technology \\ \tt{emarberg@math.mit.edu}}
\date{}

\maketitle

\begin{abstract}
Let $U_n$ denote the group of $n\times n$ unipotent upper-triangular matrices over a fixed finite field $\FF_q$, and let $U_\cP$ denote the pattern subgroup of $U_n$ corresponding to the poset $\cP$.  This work examines  the superclasses and supercharacters, as defined by Diaconis and Isaacs, of the family of normal pattern subgroups of $U_n$.  After classifying all such subgroups, we describe an indexing set for their superclasses and supercharacters given by set partitions with some auxiliary data. We go on to establish a canonical bijection between the supercharacters of $U_\cP$ and certain $\FF_q$-labeled subposets of $\cP$.  This bijection generalizes the correspondence identified by Andr\'e and Yan between the supercharacters of $U_n$  and the $\FF_q$-labeled set partitions of $\{1,2,\dots,n\}$. At present, few explicit descriptions appear in the literature of the superclasses and supercharacters of infinite families of algebra groups other than $\{U_n : n \in \NN\}$.  This work significantly expands the known set of examples in this regard.
\end{abstract}

\section{Introduction}

Consider the group $U_n$ of $n\times n$ unipotent upper triangular matrices over a finite field $\FF_q$.  Classifying this group's irreducible representations is a well-known wild problem, provably intractable for arbitrary $n$.  Despite this, Andr\'e discovered a natural way of constructing certain sums of irreducible characters and certain unions of conjugacy classes of $U_n$, which together form a useful approximation to the group's irreducible representations \cite{An95, An01, An02}. In his 
Ph.D. thesis \cite{Ya01}, Yan showed how to replace the algebraic geometry of Andr\'e's construction with more elementary methods. This simplified theory proved to have both useful applications and a natural generalization.  In particular, Arias-Castro, Diaconis, and  Stanley 
\cite{ADS04} employed Yan's work in place of the usual irreducible character theory to study random walks on $U_n$.  

Later,  Diaconis and  Isaacs \cite{DI06} axiomatized the approximating approach to define the notion of a supercharacter theory for a finite group, in which supercharacters replace irreducible characters and superclasses replace conjugacy classes.  In addition, they generalized Andr\'e's original construction to define a supercharacter theory for algebra groups, a family of groups of the form $\{1 + X : X \in \fkn \}$ where $\fkn$ is a  nilpotent (finite-dimensional, associative) $\FF_q$-algebra.  In the resulting theory, restrictions and tensor products of supercharacters decompose as nonnegative integer combinations of supercharacters. 
Furthermore, there is a notion of superinduction that is dual to restriction of supercharacters.  The references \cite{MT09, T09, T09-2, TV09} study these aspects of Diaconis and Isaacs's supercharacter theory in detail.

One of the primary motivations for these developments is the remarkable combinatorial structure of the superclasses and supercharacters of $U_n$.  Analogous to the symmetric group, where we replace partitions with set-partitions, there is a natural bijection
\be\label{class1} \barr{ccc} \left \{ \barr{c} \text{Superclasses and} \\ \text{Supercharacters of $U_n$} \earr \right\} & \leftrightarrow & \left \{ \barr{c} \text{$\FF_q$-labeled set partitions} \\ \text{of $\{1,2,\dots,n\}$} \earr \right\}\earr.\ee  Thus, Andr\'e's approximation to the representation theory of $U_n$ is not merely computable, but lends itself as a subject of interest in its own right.  Few analogues of (\ref{class1}) appear to be known for families of groups other than $U_n$.  Thiem and Venkateswaran \cite{TV09} provide one example, describing the superclasses and supercharacters of a normal series of subgroups interpolating between $U_n$ and $U_{n-1}$.   

The purpose of this work is introduce another family of examples, in particular by generalizing the classification (\ref{class1}) to all normal pattern subgroups of $U_n$.   A \emph{pattern group} is a subgroup $U_\cP \subset U_n$ of the form 
\[ U_\cP = \{ g \in U_n : g_{ij}=0\text{ if $i<j$ and }(i,j) \notin \cP \}\] where $\cP \subset \{(i,j) : 1\leq i < j \leq n\}$ is a set of positions above the diagonal.  In order for the set $U_\cP$ to form a group, $\cP$ must be a \emph{poset} on $[n] = \{1,2,\dots,n\}$; i.e., a set of positions such that $(i,j),(j,k) \in \cP$ implies $(i,k) \in \cP$.  The set of normal pattern subgroups of $U_n$ is in bijection with the set of nilpotent two-sided ideals in the algebra $\fk t_n$ of upper triangular matrices over $\FF_q$.  Both sets are naturally parametrized by Dyck paths with $2n$ steps, and hence have order $C_n = \frac{1}{n+1}\binom{2n}{n}$, the $n$th Catalan number.  We determine the superclasses and supercharacters of these groups (Theorems \ref{normal-reps} and \ref{co-normal-reps}) by constructing explicit bijections of the form
\be \label{class2} \Bigl\{ \text{Superclasses/Supercharacters of $U_\cP \vartriangleleft  U_n$} \Bigr\}
 \quad \leftrightarrow \quad  
\left\{ \barr{c} \text{$\FF_q$-labeled set partitions of} \\ \text{$[n]$ with some auxiliary data} \earr\right\}
 .\ee Unlike for $U_n$, the natural indexing sets for the superclasses and supercharacters of an arbitrary normal pattern subgroup $U_\cP\vartriangleleft U_n$ do not coincide, so we require two maps to fully define (\ref{class2}).    In the supercharacter case, the map (\ref{class2}) has a more explicit, combinatorial interpretation (Theorem \ref{bijections}) provided by the correspondence 
\be \label{class3} \Bigl\{ \text{Supercharacters of $U_\cP \vartriangleleft  U_n$} \Bigr\}
 \quad \leftrightarrow \quad  
\Bigl\{  \text{Certain $\FF_q$-labeled subposets of $\cP$}\Bigr\} .\ee  This bijection generalizes (\ref{class1}) by providing a neat combinatorial indexing set for the supercharacters of $U_\cP$, which we can defined in graph theoretic terms involving only the poset $\cP$.

Section \ref{prelim-section} provides background information on supercharacter theories, pattern groups, and relevant combinatorial constructions.  In Section \ref{normal-section} we classify all normal pattern subgroups of $U_n$ and describe a strategy for determining these groups' superclasses and supercharacters.  Sections \ref{superclass-section} and \ref{supercharacter-section} carry out this strategy to construct the bijections (\ref{class2}).  As an application, we show that if $U_\cP\vartriangleleft U_n$ then each supercharacter of $U_\cP$ is given by a product of irreducible supercharacters.  In Section \ref{posets-section}, we show how each supercharacter of $U_\cP$ corresponds to a unique labeled subposet of $\cP$, and then characterize these posets to define the correspondence (\ref{class3}).

\subsection*{Acknowledgements}

I  thank Nat Thiem and the anonymous referees for their helpful remarks and suggestions.

\section{Preliminaries}\label{prelim-section}

This section reviews the definition of an abstract supercharacter theory, then introduces a specific supercharacter theory for pattern groups.  The final subsection presents our definition of labeled set partitions, and establishes some notational conventions.

\subsection{Abstract Supercharacter Theories} \label{abstract-sc}

Let $G$ be a finite group.  We use the word \emph{character} to mean any function $G\rightarrow \CC$ of the form $g \mapsto \tr\( \rho(g)\)$ where $\rho$ is a representation of $G$ in a complex finite dimensional vector space.  As defined originally by Diaconis and Isaacs \cite{DI06}, a \emph{supercharacter theory} of $G$ is a pair $(\cS,\cS^\vee)$, where $\cS$ is a set of characters of $G$ and  $\cS^\vee$ is a partition of the elements of $G$,  satisfying the following conditions:
\begin{enumerate}
\item[(1)] $|\cS| = |\cS^\vee|$.  
\item[(2)] Each irreducible character of $G$ appears as a constituent of exactly one $\chi \in \cS$.
\item[(3)] Each $\chi \in \cS$ is constant on each set $\cK \in \cS$.
\item[(4)] The conjugacy class $\{1\} \in \cS^\vee$.
\end{enumerate} 
We call $\cS^\vee$ the set of \emph{superclasses} and $\cS$ the set of \emph{supercharacters} of the supercharacter theory $(\cS,\cS^\vee)$.  Each superclass is a union of conjugacy classes, and each supercharacter $\chi \in\cS$ is equal to a positive constant times  $\sum_{\psi  \in \Irr(\chi)} \psi(1) \psi$ where $\Irr(\chi)$ denotes the set irreducible constituents of $\chi$ \cite[Lemma 2.1]{DI06}.  By condition (2), the sets $\Irr(\chi)$ for $\chi \in \cS$ form a partition of the set $\Irr(G)$ of irreducible characters of $G$.  Consequently, the supercharacters $\cS$ form an orthogonal basis for the space of \emph{superclass functions}, the complex valued functions on $G$ which are constant on the superclasses $\cS^\vee$.

Every finite group has two supercharacter theories: the usual irreducible character theory and the trivial supercharacter theory with $\cS = \{ \One, \rho_G - \One\}$ and $\cS^\vee = \{ \{1\}, G-\{1\}\}$, where $\rho_G$ denotes the character of the regular representation of $G$.  The preprint \cite{H09} discusses several methods of constructing additional supercharacter theories of an arbitrary finite group.

In this work, we study a particular supercharacter theory introduced by Diaconis and Isaacs \cite{DI06} as a generalization of the work of Andr\'e \cite{ An95} and Yan \cite{Ya01}.  This nontrivial supercharacter theory serves as a useful approximation for the irreducible characters of groups whose representations are poorly understood, and displays some beautiful combinatorial properties.  Before introducing this supercharacter theory, we must define \emph{pattern groups,} the family of groups to which the theory applies.  This is the goal of the next section.

\subsection{Posets and Pattern Groups}

\def\cov{\mathrm{cov}}

Fix a positive integer $n$ and let $[n] = \{1,2,\dots,n\}$.  We denote by $[[n]]$ the set of positions above the diagonal in an $n\times n$ matrix:
\[ [[n]] = \{ (i,j) : 1\leq i < j \leq n\} .\]  By a \emph{poset} $\cP$ on $[n]$, we mean a subset $\cP \subset[[n]]$ such that if $(i,j), (j,k) \in \cP$ then $(i,k) \in \cP$.  
The poset 
 $\cP$ corresponds to the strict partial ordering $\prec$ of the set $[n]$ defined by setting $i \prec j$ if and only if $(i,j) \in \cP$.  

We say that $(i,k) \in \cP$ is a \emph{cover} of $\cP$ if there is no $j$ such that $(i,j),(j,k) \in \cP$, and we denote the set of covers of $\cP$ by $\cP^\cov$.    The sets $\cP$ and $\cP^\cov$ then uniquely determine each other: namely, 
\be\label{covers2posets} (i,k) \in \cP\quad \text{if and only if}\quad \exists\ (j_1,j_2), (j_2, j_3), \dots , (j_{r-1}, j_r) \in \cP^\cov\text{ with $i=j_1$ and $k=j_r$.}\ee  We can visually depict $\cP$ via its Hasse diagram, which is the directed graph whose vertices are $1,2,\dots,n$ and whose directed edges are the ordered pairs $(i,j)\in \cP^\cov$.  For example, we can define the poset $\cP = \{ (1,3), (1,4), (2,3), (2,4), (3,4) \}$ on $n=4$ by writing
\[ \cP\ =\  \xy<0.25cm,0.8cm> \xymatrix@R=.3cm@C=.3cm{
  &  4  \\
  &  3  \ar @{-} [u]   \\
1 \ar @{-} [ur]   & 2 \ar @{-} [u] 
}\endxy \]

Fix a finite field $\FF_q$ with $q$ elements, and let $U_n$ denote the group of $n\times n$ upper-triangular matrices with ones on the diagonal and entries in $\FF_q$.  Given a poset $\cP$ on $[n]$, the \emph{pattern group} $U_\cP$ is the subgroup $U_n$ given by 
\[ U_\cP = \{ g \in U_n : g_{ij} =0\text{ if $i<j$ and $(i,j)\notin \cP$} \}.\]  
The group $U_n$ is the pattern group corresponding to the poset $\cP = [[n]]$.   

\begin{remarks}
\begin{enumerate}
\item[]
\item[(a)] Note that under our definitions, the set $\cP$ can serve as a poset on $[n]$ for any sufficiently large integer $n$.  Thus, implicit in the notation $U_\cP$ is the choice of a dimension $n$ corresponding to $\cP$.   Different choices of $n$ result in canonically isomorphic pattern groups, however.
\item[(b)] Examples of pattern groups include the unipotent radicals of rational parabolic subgroups of the finite general linear groups $\GL_n(\FF_q)$.  One can describe many group theoretic structures of pattern groups in terms of the poset $\cP$, such as the center, Frattini subgroups, coset representatives, etc.\ (see \cite{DT09,MT09} for examples).
 \end{enumerate}
\end{remarks}

\subsection{Superclasses}

Let $\fkn_n$ denote the nilpotent $\FF_q$-algebra of strictly upper triangular $n\times n$ matrices with entries in $\FF_q$.  For any matrix $X$, let 
\[\supp(X) = \{ (i,j) : X_{ij} \neq 0\}\] denote the set of positions in $X$ with nonzero entries.  Now, given a poset $\cP$ on $[n]$, define $\fkn_\cP \subset \fkn_n$ as the nilpotent $\FF_q$-algebra 
\[ \fkn_\cP = U_\cP - 1 =  \{ X \in \fkn_n : \supp(X)\subset \cP \}.\]  

The group $U_\cP$ acts on the algebra $\fkn_\cP$ on the left and right by multiplication.  The map $X\mapsto 1 +X$ gives a bijection $\fkn_\cP \rightarrow U_\cP$, and we define the \emph{superclasses} of $U_\cP$ to be the sets formed by applying this map to the two-sided $U_\cP$-orbits in $\fkn_\cP$.  The superclass of $U_\cP$ containing $g \in U_\cP$, denoted $\cK_\cP^g$, is therefore the set
\[ \label{superclass-defin} \cK_\cP^g = \{ 1+x(g-1)y : x,y \in U_\cP\}.\]  Each superclass is a union of conjugacy classes, and one superclass consists of just the identity element of $U_\cP$.

\subsection{Supercharacters}

Given a poset $\cP$ on $[n]$, let $\fkn_\cP^*$ denote the dual space of $\FF_q$-linear functionals $\lambda : \fkn_\cP \rightarrow \FF_q$.  Similarly, let $\fkn_n^*$ denote the dual space of $\fkn_n$.  Throughout this work, we distinguish linear functionals by their domains, and hence do not view $\lambda \in \fkn_\cQ^*$ as an element of $\fkn_\cP^*$ when $\fkn_\cP\subset \fkn_\cQ$.  This precaution avoids some potential ambiguities later.

One of the benefits of working with pattern groups is that we have a canonical way of identifying $\fkn_\cP^*$ with $\fkn_\cP$. Specifically, we associate $\lambda \in \fkn_\cP^*$ with the matrix in $\fkn_\cP$ whose $(i,j)$th entry is $\lambda_{ij}$, where we define
\[ \lambda_{ij} = \begin{cases} \lambda(e_{ij}), & \text{if $(i,j) \in \cP$}, \\ 0,&\text{otherwise}.\end{cases}\]
Here $e_{ij}$ denotes the elementary  $n\times n$ matrix with 1 in entry $(i,j)$ and 0 in all other entries.   This identification gives a vector space isomorphism $\fkn_\cP^* \cong \fkn_\cP$, and provides a convenient way of defining linear functionals on $\fkn_\cP$ as matrices.  Following the convention for matrices, given $\lambda \in \fkn_\cP^*$, we let \[\supp(\lambda) = \{ (i,j)\in \cP : \lambda_{ij} \neq 0\}.\]  

We have left and right actions of the group $U_\cP$ on the vector space $\fkn_\cP^*$ given by defining $g\lambda$ and $\lambda g$ for $g \in U_\cP$ and $\lambda \in \fkn_\cP^*$ to be the functionals with
\[\label{star-action} g\lambda(X) = \lambda(g^{-1}X)\qquad\text{and}\qquad \lambda g(X) = \lambda(Xg^{-1})\qquad\text{for }X \in \fkn_\cP.\] These actions are compatible in the sense that  $(g\lambda)h = g(\lambda h)$ for $g,h \in U_\cP$, and $\lambda \in \fkn_\cP^*$.  Hence  we may remove all parentheses without introducing ambiguity.  Given $\lambda \in \fkn_\cP^*$, we denote the corresponding left, right, and two-sided $U_\cP$-orbits by $U_\cP\lambda$, $\lambda U_\cP$, $U_\cP\lambda U_\cP$.

Fix a nontrivial group homomorphism $\theta : \FF_q^+\rightarrow \CC^\times$.  The \emph{supercharacters} of $U_\cP$ are the functions $\chi^\lambda_\cP : U_\cP \rightarrow \CC$ indexed by $\lambda \in \fkn_\cP^*$, defined by the formula  
\be\label{formula}\chi^\lambda_\cP(g) = \frac{|U_\cP \lambda |}{|U_\cP \lambda U_\cP|} \sum_{\mu \in U_\cP\lambda U_\cP} \theta\circ \mu(g-1),\qquad\text{for }g \in U_\cP.\ee 
It follows from this definition that supercharacters are constant on superclasses.  In addition, we have $\chi_\cP^\lambda = \chi_\cP^\mu$ if and only if $\mu \in U_\cP \lambda U_\cP$.  $\chi_\cP^\lambda$ is the character of the left $U_\cP$-module 
\[ V^\lambda_\cP = \CC\spanning\{ v_\mu : \mu \in U_\cP \lambda \},\qquad\text{where }gv_\mu = \theta\circ \mu\(1-g^{-1}\) v_{g\mu}\text{ for }g\in U_\cP.\]  
For $\lambda, \mu \in \fkn_\cP^*$, 
\[ \langle \chi^\lambda_\cP, \chi^\mu_\cP \rangle_{U_\cP} =\begin{cases} |U_\cP\lambda \cap \lambda U_\cP|,&\text{if }\mu \in U_\cP\lambda U_\cP, \\ 0,&\text{otherwise,}\end{cases} 
\quad\text{where }\langle \chi, \psi \rangle_{U_\cP} = \frac{1}{|U_\cP|} \sum_{g\in U_\cP} \chi(g)\overline{\psi(g)}.\] 
Thus, $\chi_\cP^\lambda$ is irreducible if and only if $U_\cP \lambda \cap \lambda U_\cP  = \{\lambda\}$, and distinct supercharacters are orthogonal.

\begin{notation}
If the context is clear, we may drop the subscript and write $\cK^g$, $\chi^\lambda$ to denote the superclass and supercharacter $\cK^g_\cP$, $\chi_\cP^\lambda$.
\end{notation}

\begin{remarks}
\begin{enumerate}
\item[]
\item[(a)] Diaconis and Isaacs \cite{DI06} first defined this set of superclasses and supercharacters for a larger family of groups known as \emph{algebra groups}, generalizing the work of Andr\'e and Yan \cite{An95,Ya01} which applied only to $U_n$.  Diaconis and Thiem derive a more explicit supercharacter formula in \cite{DT09}.

\item[(b)]    While \cite{DI06} defines the module $V^\lambda_\cP$ abstractly, we can construct it as an explicit submodule of $\CC U_\cP$ by defining
$v_\mu  =  \sum_{x \in U_\cP} \theta\circ \mu(1-x)x \in \CC U_\cP$ for $\mu \in \fkn_\cP^*$.  

\item[(c)] The numbers of superclasses and supercharacters are equal to the numbers of two-sided $U_\cP$ orbits in $\fkn_\cP$ and $\fkn_\cP^*$, and these are the same by Lemma 4.1 in \cite{DI06}.  Furthermore, it is clear from the formula (\ref{formula}) that the character $\rho_{U_\cP}$ of the regular representation of $U_\cP$ decomposes as 
\[ \rho_{U_\cP} = \sum_\lambda \frac{|U_\cP \lambda U_\cP|}{|U_\cP \lambda|} \chi_\cP^\lambda\] where the sum is over a set of representatives $\lambda$ of the two-sided $U_\cP$ orbits in $\fkn_\cP^*$.  Thus, each irreducible character of $U_\cP$ appears as a constituent of a unique supercharacter.  Consequently, the supercharacters and superclasses defined above indeed form a supercharacter theory of $U_\cP$ in the sense of Section \ref{abstract-sc}.

\end{enumerate}
\end{remarks}

\subsection{$\FF_q$-Labeled Set Partitions and the Supercharacters of $U_n$}\label{sect2}

The supercharacter theory described in the preceding section arose as a generalization of a specific attempt to approximate the irreducible characters of $U_n$.  The classification of this group's conjugacy classes and irreducible representations is  a wild problem, but the classification of its superclasses and supercharacters has a highly satisfactory combinatorial answer in terms of \emph{$\FF_q$-labeled set partitions.}  We define these objects below, and then describe how they correspond to the superclasses and supercharacters of $U_n$.

Fix a nonnegative integer $n$.  A \emph{set partition} $\lambda = \{\lambda_1, \lambda_2, \dots, \lambda_\ell\}$ of $[n]$ is a set of disjoint, nonempty sets $\lambda_i \subset \{1,\dots,n\}$ such that $\bigcup_i \lambda_i = [n]$. The sets $\lambda_i$ are called the \emph{parts} of $\lambda$.  We view each part as a finite increasing sequence of positive integers, and typically abbreviate $\lambda$ by writing the numbers in each part from left to write, separating successive parts with the ``$|$'' symbol.  For example, we write $\lambda = \{\{1,2\}, \{3\},\{4,7,8\}, \{5,6\}\}$ as $\lambda = 12|3|478|56$.  

The \emph{support} of a set partition $\lambda = (\lambda_1,\lambda_2,\dots,\lambda_\ell)$ is the set 
\[ \supp(\lambda) = \{ (i,j) : \text{$i<j$ and for some $k$, we have $i,j \in \lambda_k$ and $i<x<j$ only if $x \notin \lambda_k$}\}.\] In other words, $(i,j) \in \supp(\lambda)$ if and only if $i<j$ are consecutive integers in some part of $\lambda$.  For example, the support of $\lambda = 12|3|478|56$ is $\supp(\lambda) = \{(1,2), (4,7),(7,8),(5,6)\}$.  The set $\supp(\lambda)$ is the same as the set $\cA(\lambda)$ defined in \cite{T09-2} and $A(\lambda)$ defined in \cite{T09}.

An \emph{$\FF_q$-labeled set partition} is a set partition $\lambda$ with a map $\supp(\lambda) \rightarrow \FF_q^\times$ which labels each element of the support with a nonzero element of $\FF_q$.  We represent a labeled set partition by writing the set partition $\lambda$ as above, and then replacing each supported point ``$ij$'' with ``$i\larc{t}j$'' where $t \in \FF_q^\times$ is the label assigned to $(i,j)$.  \
For example, the $\FF_q$-labeled set partitions corresponding to the (unlabeled) set partition $12|3|478|56$ are of the form
 \[\lambda = 1\larc{a}2|3|4\larc{b}7\larc{c}8|5\larc{d}6,\qquad\text{where }a,b,c,d \in \FF_q^\times.\]
For each $(i,j) \in \supp(\lambda)$, let $\lambda_{ij} \in \FF_q^\times$ denote the corresponding label, and for each $(i,j) \notin \supp(\lambda)$ let $\lambda_{ij} = 0$.  This notation naturally assigns to a labeled set partition $\lambda$ of $[n]$ a strictly upper triangular $n\times n$ matrix over $\FF_q$; namely, the matrix whose $(i,j)$th entry is $\lambda_{ij}$.  For example, 
\[ 1\larc{r}3\larc{s}5|2\larc{t}4 \quad\text{corresponds to}\quad \(\begin{array}{ccccc} 
0 & 0 & r & 0 & 0 \\
  & 0 & 0 & t & 0 \\
  &   & 0 & 0 & s \\
  &   &   & 0 & 0 \\ 
  &   &   &   & 0 
\end{array}\).\]  This correspondence defines a bijection
\[ \barr{ccc} \left\{ 
\barr{c} 
\text{$\FF_q$-labeled set} \\ \text{partitions of $n$}
\earr
\right\} & \leftrightarrow &
\left\{ 
\barr{c} 
\text{Matrices in $\fkn_n$ with at most one nonzero} \\ \text{entry in each row and column. }
\earr
\right\}.
\earr\]

 \def\sP{\mathscr{S}}
 
We can view the $n\times n$ upper triangular matrix defined by a labeled set partition of $[n]$ as an element of either the algebra $\fkn_{n}$ or  the dual space $\fkn_{n}^*$.  To distinguish between these two identifications, we adopt the following notation: let 
\[ \ba \sP_n &= \{ X \in \fkn_n : \supp(X)\text{ contains at most one position in each row and column} \}, \\
 \sP_n^* &= \{ \lambda \in \fkn_n^* : \supp(\lambda)\text{ contains at most one position in each row and column} \}.\ea\] 
 More generally, given any poset $\cP$ on $[n]$, define 
 \[ \ba \sP_\cP &= \{ X \in \fkn_\cP : \supp(X)\text{ contains at most one position in each row and column} \}, \\
 \sP_\cP^* &= \{ \lambda \in \fkn_\cP^* : \supp(\lambda)\text{ contains at most one position in each row and column} \}.\ea\] We refer to elements of both of these sets as $\FF_q$-labeled set partitions of $[n]$.  Observe that the support of a set partition is well-defined and consistent, in the sense that $\supp(\lambda)$ defines the same set, whether $\lambda$ is viewed as a set partition, a matrix, or a linear functional.
  
Yan showed in \cite{Ya01} that the superclasses and supercharacters of $ U_n$ are indexed by the set of all $\FF_q$-labeled set partitions of $[n]$.  In particular, the maps
\be\label{U_n-classification} \barr{ccc}\sP_n & \to &  \Bigl\{ \text{Superclasses of $ U_n$} \Bigr\} \\ 
\lambda & \mapsto & \cK_{}^{1+\lambda} \earr
\qquad\text{and}\qquad
\barr{ccc}\sP_{n}^* & \to &  \Bigl\{\text{Supercharacters of $ U_n$} \Bigr\} \\ 
\lambda & \mapsto & \chi_{}^\lambda. \earr
\ee  are bijections.  Andr\'e proved the character result earlier from a more geometric perspective in \cite{An95}.  The indexing sets $\sP_n$ and $\sP_n^*$ provide the following  simple supercharacter formula for $U_n$:
\[ \chi^\lambda(1 + \mu) =\begin{cases} \displaystyle \prod_{(i,l) \in \supp(\lambda)} \frac{q^{l-i-1} \theta(\lambda_{il} \mu_{il})}{q^{|\{ (j,k) \in \supp(\mu) : i<j<k<l \}|}}  ,&\barr{l} \text{if }(i,j),(j,k) \notin\supp(\mu)\text{ whenever } \\ \text{$i<j<k$ and $(i,k) \in \supp(\lambda)$,}\earr  \\ 0,&\text{otherwise,}\end{cases}\]
for $\lambda \in \sP_n^*$ and $\mu \in \sP_n$.  It is evident from this formula that for each $\lambda \in \sP_n^*$, the supercharacter $\chi^\lambda$ of $U_n$ factors as a product of irreducible supercharacters indexed by $\nu \in \sP_n^*$ with $|\supp(\nu)| \leq 1$. 

The primary intent of this work is to generalize this classification to  all normal pattern subgroups of $U_n$.  In order to do this, we first require some understanding of what such pattern groups look like; the next section provides this information.

\section{Normal Pattern Subgroups of $ U_n$}\label{normal-section}

Given a poset $\cP$ on $[n]$, we say that $\cP$ is normal in $[[n]]$ and write $\cP\vartriangleleft [[n]]$ if 
\be\label{normal-def} \text{$(i,l) \notin \cP$ implies  $(j,k)\notin \cP$,}\qquad\text{for all $1\leq i\leq j < k \leq l \leq n$.}\ee  Taking the contrapositive, $\cP \vartriangleleft [[n]]$ if and only if $(j,k) \in \cP$ implies $(i,l)\in \cP$ for all $1\leq i \leq j < k \leq l\leq n$.  Of course, the reason for adopting this notation has much to do with the following.

\begin{lemma}\label{normal-in-U_n} If $\cP$ is a poset on $[n]$, then $U_\cP \vartriangleleft  U_n$ 
if and only if $\cP\vartriangleleft[[n]]$.
 \end{lemma}

This result appears in a more general form as Lemma 4.1 in \cite{M09}.  Its proof leads to the following pair of corollaries, which appear as Lemma 3.2 in \cite{MT09} and Corollary 4.1 in \cite{M09}.



\begin{corollary}\label{supernormal}
If $\cP$ is poset with $\cP\vartriangleleft [[n]]$, then $gXh \in \fkn_\cP$ for all $g,h \in U_n$ and $X \in \fkn_\cP$.  Thus $\fkn_\cP$ is a two-sided ideal in $\fkn_n$ and $U_\cP$ is a union of superclasses of $U_n$.
\end{corollary}

\def\fkt{\mathfrak{t}}

This property  will allow us to use the classification (\ref{U_n-classification}) to great advantage in our analysis of the actions of $U_\cP$ on $\fkn_\cP$ and $\fkn_\cP^*$.  For the second corollary, let $\fkt_n$ denote the algebra of $n\times n$ upper triangular matrices over $\FF_q$.

\begin{corollary} A subset $\fka \subset \fkt_n$ is a nilpotent two-sided ideal if and only if $\fka = \fkn_\cP$ for some poset $\cP\vartriangleleft [[n]]$.
\end{corollary}


Proposition 2 in \cite{S75} shows the number of nilpotent two-sided ideals in $\fkt_n$, and hence the number of normal pattern subgroups of $U_n$, to be the $n$th Catalan number $C_n = \frac{1}{n+1}\binom{2n}{n}$.  In fact, this result holds if $\fkt_n$ is taken to be the algebra of $n\times n$ upper triangular matrices over any field.  Intuitively, this follows by viewing the positions of a matrix as the interior squares of an $n\times n$ grid.  Then, given a normal poset $\cP\vartriangleleft [[n]]$, consider the border separating the positions in $\cP$ from all other positions on and above the diagonal.  Condition (\ref{normal-def}) ensures that this border is a monotonic path in the grid starting at the upper left hand corner, ending at the lower right corner, and never passing below the diagonal.  For example,
\[ U_\cP = \left\{ \(\barr{cccc} 1 & * & * & *  \\ & 1 & 0 & *   \\ & & 1 & 0   \\ & & & 1  \earr\) \right\} \vartriangleleft U_4 \qquad\text{corresponds to}\qquad 
\xy<0.45cm,1.35cm> \xymatrix@R=.3cm@C=.3cm{
\bullet \ar @{-} [r] \ar @{.} [rrrr]& \bullet   & \cdot  &\cdot  &\cdot  \\
\cdot   \ar @{.} [rrrr]& \bullet \ar @{-} [u] \ar @{-} [r]  & \bullet \ar @{-} [r]  & \bullet  &\cdot \\
\cdot \ar @{.} [rrrr]  & \cdot & \cdot & \bullet \ar @{-} [u] \ar @{-} [r] & \bullet   \\
\cdot \ar @{.} [rrrr]  & \cdot & \cdot & \cdot &    \bullet \ar @{-} [u] \\
\cdot \ar @{.} [rrrr] \ar @{.} [uuuu]  & \cdot \ar @{.} [uuuu]& \cdot \ar @{.} [uuuu]& \cdot \ar @{.} [uuuu]& \bullet\ar @{.} [uuuu] \ar @{-} [u]  
}\endxy\]
This gives a bijective correspondence between the set nilpotent ideals in $\fkt_n$ and the set of Dyck paths of order $n$, which has order $C_n$.  This number is independent of  $\FF_q$.  By contrast, the number of nilpotent two-sided ideals in $\fkn_n$ depends significantly on the field $\FF_q$, and is infinite if $\FF_q$ is replaced by a field of characteristic zero.  See Section 4 in \cite{M09} for a more detailed discussion.

Now that we have in  classified the normal pattern subgroups of $U_n$, we can set about describing the family's superclasses and supercharacters.  Fix a poset $\cP\vartriangleleft [[n]]$.  Corollary \ref{supernormal} then shows that $U_n$ acts on $\fkn_\cP$ on the left and right by multiplication.  These actions in turn give rise to compatible left and right actions of $U_n$ on $\fkn_\cP^*$, defined in the usual way by $g\lambda(X) = \lambda(g^{-1}X)$ and $\lambda g(X) = \lambda(Xg^{-1})$ for $g \in U_n$, $\lambda \in \fkn_\cP^*$, and $X \in \fkn_\cP$.  Since $gU_\cP = gU_\cP g^{-1} g = U_\cP g$ for all $g \in U_n$, we can view $U_n$ as acting (on the left and right) on the left, right, and two-sided $U_\cP$ orbits of $\fkn_\cP$ and $\fkn_\cP^*$.  For example we have 
\[ gU_\cP X U_\cP h  = U_\cP(gXh)U_\cP\qquad\text{and}\qquad g U_\cP \lambda U_\cP h = U_\cP (g\lambda h) U_\cP\] for $g,h \in U_n$, $X \in \fkn_\cP$, and $\lambda \in \fkn_\cP^*$.  These actions evidently preserve all orbit sizes, so it follows that each left/right/two-sided $U_n$-orbit in $\fkn_\cP$ or $\fkn_\cP^*$ decomposes as a disjoint union of left/right/two-sided $U_\cP$-orbits, all of which have the same cardinality.

This last statement suggests a strategy for identifying the superclasses and supercharacters of the normal pattern subgroup $U_\cP$.  This classification amounts to describing the two-sided $U_\cP$-orbits in $\fkn_\cP$ and $\fkn_\cP^*$, and we can do this in two steps: by first finding the $U_n$-orbits in $\fkn_\cP$ and $\fkn_\cP^*$, and then decomposing each $U_n$-orbit into $U_\cP$-orbits.  The first step in this process is in some sense trivial, since by (\ref{U_n-classification}) we can index the orbits of the action of $U_n$ with $\FF_q$-labeled set partitions.  We accomplish the second step by introducing some additional constructions defined in terms of the poset $\cP$ to identify the distinct $U_\cP$-orbits in each $U_n$-orbit.

Of course, these ideas apply equally well to the problem of describing the superclasses 
 and supercharacters of the normal pattern subgroups of an arbitrary pattern group $U_\cQ$ in place of  $U_n$.  However, one  needs to thoroughly understand the superclasses and supercharacters of $U_\cQ$ to derive anything very explicit about the analogous structures for the group's normal  subgroups, which is why we restrict our attention to the case $U_\cQ = U_n$.
 
We carry out the strategy described above in the next two sections, then derive an additional correspondence between the supercharacters of $U_\cP\vartriangleleft U_n$ and certain subposets of $\cP$ whose covers are labeled by elements of $\FF_q^\times$.

\section{Superclass Constructions}\label{superclass-section}

Fix a poset $\cP\vartriangleleft [[n]]$.  We know that the two-sided $U_n$-orbits in $\fkn_\cP$ are indexed by the set partitions $\sP_\cP$, and we need a way of identifying the distinct $U_\cP$ orbits in each.  We accomplish this with the following constructions.  Given a set partition $\lambda \in \sP_{\cP}$ and a poset $\cQ \in \{ \cP, [[n]]\}$,  
define 
\be \label{adj} \ba \adj_\cQ^{\mathrm{L}}(\lambda) &= \{ (i,k) : \exists\ (j,k) \in \supp(\lambda)\text{ with }(i,j) \in \cQ \},\\
\adj_\cQ^{\mathrm{R}}(\lambda) &= \{ (i,k) : \exists\ (i,j) \in \supp(\lambda)\text{ with }(j,k) \in \cQ \}, \\
\adj_\cQ(\lambda) &= \adj_\cQ^{\mathrm{L}}(\lambda) \cup \adj_\cQ^{\mathrm{R}}(\lambda).
\ea \ee
The notation $\adj$ comes from thinking of these positions as being \emph{adjacent} to the support of $\lambda$ with respect to $\cQ$.  In particular, these positions are the only ones which we can ``reach'' by acting on $\lambda$ with $U_\cQ$, in the sense that elements in the one-sided orbits $U_\cQ \lambda$ and $\lambda U_\cQ$ have nonzero entries only in $\supp(\lambda)$ and in $\adj_\cQ^{\mathrm{L}}(\lambda)$ and $\adj_\cQ^{\mathrm{R}}(\lambda)$, respectively.

It is clear from our definition of a poset and Lemma \ref{normal-in-U_n} that the positions defined by $\adj_\cQ^{\mathrm L}$ and $\adj_\cQ^{\mathrm R}$ are contained in $\cP$ regardless of whether $\cQ = \cP\text{ or }\cQ=[[n]]$.  
Next define the set differences
\be\label{aux} 
 \ba \aux^{\mathrm{L}}_\cQ(\lambda) &= \adj_{[[n]]}^{\mathrm L} (\lambda) -  \adj_{\cQ}^{\mathrm{L}}(\lambda)
 ,\\
\aux^{\mathrm{R}}_\cQ(\lambda) &= \adj_{[[n]]}^{\mathrm R} (\lambda) -  \adj_{\cQ}^{\mathrm{R}}(\lambda)
,\\
\aux_\cQ(\lambda) &= \adj_{[[n]]}^{} (\lambda) -  \adj_{\cQ}^{}(\lambda).
\ea
\ee 
The notation $\aux$ comes from viewing these positions as \emph{auxiliary} to the support of $\lambda$ with respect to $\cQ$, in the sense that the entries in these positions comprise the minimum amount of information necessary to specify the distinct $U_\cQ$-orbits in $U_n \lambda U_n$.  These sets of positions are likewise always subsets of $\cP$; in particular, note that they are empty when $\cQ = [[n]]$. 

The sets in both (\ref{adj}) and (\ref{aux}) are all disjoint from $\supp(\lambda)$.  This follows since $\supp(\lambda)$ contains at most one position in each row and column as  $\lambda$ is a set partition, and the positions in (\ref{adj}) and (\ref{aux}) each lie either in the same row and strictly to the right of a position in $\supp(\lambda)$, or in the same column and strictly above a position in $\supp(\lambda)$.

\begin{example}\label{superclass-example}
 Suppose $\cP\vartriangleleft[[7]]$ is the poset given by
\[\cP\ =\ \ \xy<0.45cm,1.15cm> \xymatrix@R=.3cm@C=.3cm{
7    &    \\
5 \ar @{-} [u]     & 6  \\
3\ar @{-} [u] \ar @{-} [ur]    & 4\ar @{-} [u]\ar @{-} [uul]  \\
1 \ar @{-} [u] \ar @{-} [ur]    & 2\ar @{-} [u]\ar @{-} [uul] \\
}\endxy \qquad\text{and}\qquad \lambda =1\larc{a}4\larc{b}6|2\larc{c}5|3\larc{d}7 
= \(\begin{array}{ccccccc} 
0 & 0 & 0 & a & 0 & 0 & 0 \\
 & 0 & 0 & 0 & c & 0 & 0 \\
 &  & 0 & 0 & 0& 0 & d \\
 &  &  & 0 & 0 & b & 0 \\
 &  &  &  & 0 & 0 & 0 \\
 &  &  &  &  & 0 & 0 \\
 &  &  &  &  &  & 0 \end{array}\) \in \sP_\cP.
\]
Then $U_\cP$ is the commutator subgroup of $U_{7}$, which is explicitly given by the set of matrices $x \in U_{7}$ with $x_{i,i+1} = 0$ for all $i$.  We compute \[ \ba 
&\adj_{[[7]]}^{\mathrm{L}}(\lambda) = \left\{ \barr{l} (1,5),(1,6),(1,7), \\ (2,6),(2,7), (3,6)\earr\right\}, \\ 
&\adj_\cP^{\mathrm{L}}(\lambda) = \{ (1,6),(1,7),(2,6) \},\\ 
&\aux_\cP^{\mathrm{L}}(\lambda) = \{ (1,5),(2,7), (3,6)\},
\ea
\qquad\text{and}\qquad 
\ba
&\adj_{[[7]]}^{\mathrm{R}}(\lambda) = \left\{\barr{l} (1,5),(1,6),(1,7), \\ (2,6),(2,7), (4,7)\earr\right\}, \\ 
&\adj_\cP^{\mathrm{R}}(\lambda) = \{ (1,6),(1,7),(2,7) \},\\ 
&\aux_\cP^{\mathrm{R}}(\lambda) = \{ (1,5),(2,6),(4,7)\} .
\ea\]  
Thus,  $\adj_{[[7]]}(\lambda)$, $\adj_{\cP}(\lambda)$, and $\aux_\cP(\lambda)$ are the following sets of positions in a $7\times 7$ matrix:
\[ \barr{c} \(\begin{array}{ccccccc} 
0 & 0 & 0 & {\color{Gray}a} & {\color{Green}\blacksquare} & {\color{Green}\blacksquare}  & {\color{Green}\blacksquare}  \\
 & 0 & 0 & 0 & {\color{Gray}c} & {\color{Green}\blacksquare}  & {\color{Green}\blacksquare}  \\
 &  & 0 & 0 & 0& {\color{Green}\blacksquare}  & {\color{Gray}d}  \\
 &  &  & 0 & 0 & {\color{Gray}b} & {\color{Green}\blacksquare}  \\
 &  &  &  & 0 & 0 & 0 \\
 &  &  &  &  & 0 & 0 \\
 &  &  &  &  &  & 0 \end{array}\), \\
 {\color{Green}\blacksquare} \in \adj_{[[7]]}(\lambda) \earr 
 \quad  \barr{c} \(\begin{array}{ccccccc} 
0 & 0 & 0 & {\color{Gray}a} & 0 & {\color{Blue}\blacksquare}  & {\color{Blue}\blacksquare}  \\
 & 0 & 0 & 0 & {\color{Gray}c} & {\color{Blue}\blacksquare}  & {\color{Blue}\blacksquare}  \\
 &  & 0 & 0 & 0& 0  & {\color{Gray}d}  \\
 &  &  & 0 & 0 & {\color{Gray}b} & 0  \\
 &  &  &  & 0 & 0 & {0} \\
 &  &  &  &  & 0 & 0 \\
 &  &  &  &  &  & 0 \end{array}\), \\
 {\color{Blue}\blacksquare} \in \adj_\cP(\lambda) \earr 
 \quad
\barr{c}  \(\begin{array}{ccccccc} 
0 & 0 & 0 & {\color{Gray}a} & {\color{Red}\blacksquare} & 0  & 0  \\
 & 0 & 0 & 0 & {\color{Gray}c} & 0 & 0  \\
 &  & 0 & 0 & 0& {\color{Red}\blacksquare}  & {\color{Gray}d}  \\
 &  &  & 0 & 0 & {\color{Gray}b} & {\color{Red}\blacksquare}  \\
 &  &  &  & 0 & 0 & 0  \\
 &  &  &  &  & 0 & 0 \\
 &  &  &  &  &  & 0 \end{array}\). \\ 
 {\color{Red}\blacksquare} \in \aux_\cP(\lambda) \earr  \]
 \end{example}
 
The main result of this section, and the singular motivation for these definitions, is the construction of a bijection
\[ \barr{ccc}  
\Bigl\{
\text{Superclasses of $U_\cP$}
\Bigr\}
 & \leftrightarrow & 
 \Bigl\{
(\lambda, X) :  \text{$\lambda \in \sP_\cP$ and $X \in \fkn_\cP$, $\supp(X) \subset \aux_\cP(\lambda)$} 
\Bigr\}.
\earr\]
In this direction, we first have the following lemma, which employs (\ref{adj}) and (\ref{aux}) to classify the one-sided $U_\cP$-orbits in $\fkn_\cP$.

\begin{lemma} \label{one-sided-orbits} Fix a poset $\cP\vartriangleleft [[n]]$ and a set partition $\lambda \in \sP_\cP$, and let $\cQ \in \{\cP, [[n]]\}$.  
\begin{enumerate}
\item[(a)] Each left $U_\cQ$-orbit in $U_n\lambda$ has the form
 $
 \left\{ \lambda +X + Z \in \fkn_\cP: \supp(Z) \subset \adj_\cQ^{\mathrm L}(\lambda)\right\}
 $ for a unique matrix $X \in \fkn_\cP$ with $\supp(X)\subset \aux_\cQ^{\mathrm L}(\lambda)$.
 
 \item[(b)]  Each right $U_\cQ$-orbit in $\lambda U_n$ has the form
$ \left\{ \lambda + Y + Z  \in \fkn_\cP : \supp(Z) \subset \adj_\cQ^{\mathrm R}(\lambda)\right\}
$ for a unique matrix $Y \in \fkn_\cP$ with $\supp(Y) \subset \aux_\cQ^{\mathrm R}(\lambda)$.  

\item[(c)] Consequently, $|U_\cQ \lambda U_\cQ| = q^{|\adj_\cQ(\lambda)|}$.
\end{enumerate}
\end{lemma}

\begin{proof}
We only prove (a) as (b) follows by similar arguments and (c) is immediate from the first two parts since $|U_\cQ\lambda U_\cQ| = \frac{|U_\cQ\lambda| |\lambda U_\cQ|}{|U_\cQ\lambda \cap \lambda U_\cQ|}$ \cite[Lemma 3.1]{DI06}.
  In this direction, we first show that 
\[U_\cQ \lambda = \left\{\lambda+ Z \in \fkn_\cP: \supp(Z) \subset \adj_\cQ^{\mathrm L}(\lambda)\right\}.\]   
This holds since $(g\lambda-\lambda)_{ik} =  \sum_{i<j<k} g_{ij} \lambda_{jk}$ for $g\in U_\cQ$.  If $(i,k)\notin \adj_\cQ^{\mathrm{L}}(\lambda)$ then either $(j,k)\notin \supp(\lambda)$ for all $j>i$, or  $(i,j)\notin \cQ$ whenever $(j,k) \in\supp(\lambda)$.  In both cases $(g\lambda-\lambda)_{ik} = 0$, so  $\supp(g\lambda-\lambda) \subset \adj_\cQ^{\mathrm L}(\lambda)$ and $U_\cQ\lambda$ is contained in the given right hand set.  To show the reverse containment, we observe that $U_\cQ\lambda -\lambda$ is a vector space, and that if $(i,k) \in \adj_\cQ^{\mathrm L}(\lambda)$ then we have $(i,j) \in \cQ$ and $(j,k) \in \supp(\lambda)$, so $e_{ik} = (1+\lambda_{jk}^{-1}e_{ij}) \lambda - \lambda \in U_\cQ \lambda - \lambda$.  Thus the right hand set is itself contained in $U_\cQ\lambda$.  

Since $\aux_{[[n]]}^{\mathrm L}(\lambda) = \varnothing$, this proves the lemma when $\cQ = [[n]]$.  To treat the case $\cQ = \cP$, observe that 
\[\frac{| U_n\lambda|}{|U_\cP\lambda|} = q^{|\adj_{[[n]]}^{\mathrm L}(\lambda)| - |\adj_\cP^{\mathrm L}(\lambda)|} = q^{|\aux_\cP^{\mathrm L}(\lambda)|},\]
 and so the number of left $U_\cP$-orbits in $ U_n\lambda$ is the same as the number of elements $X \in \fkn_\cP$ with $\supp(X)\subset \aux_\cP^{\mathrm L}(\lambda)$.  It therefore suffices to demonstrate that $U_\cP(\lambda + X) = U_\cP\lambda +X$ for $X \in \fkn_\cP$ with $\supp(X) \subset \aux_\cP^{\mathrm L}(\lambda)$, as this shows that distinct elements of the form $\lambda+X$ belong to distinct left $U_\cP$-orbits, and that these orbits are of the desired form.  To this end, fix $X\in \fkn_\cP$ with $\supp(X) \subset \aux_\cP^{\mathrm L}(\lambda)$.   Let $g \in U_\cP$, and note that $(gX-X)_{ik} = \sum_{i<j<k} g_{ij}X_{jk}$.  If $(i,k) \in \aux_\cP^{\mathrm L}(\lambda)$, then by definition there is some $i<j'<k$ with $(j',k) \in \supp(\lambda)$ and $(i,j')\notin \cP$.  In this case, for each $i<j<k$ either $X_{jk} = 0$ or 
\[(j,k) \in \aux_\cP^{\mathrm L}(\lambda) \Rightarrow (j,j')\in [[n]]- \cP \Rightarrow (i,j) \notin \cP \Rightarrow g_{ij} = 0\] by Lemma \ref{normal-in-U_n}.  Hence $(gX-X)_{ik} = 0$ if $(i,k) \in \aux_\cP^{\mathrm L}(\lambda)$.  Clearly $(gX-X)_{ik} = 0$ if $(i,k) \notin \adj_{[[n]]}^{\mathrm L}(\lambda)$ since in this case we have $X_{jk} = 0$ for all $j>i$.  Thus, $\supp(gX-X)\subset \adj_\cP^{\mathrm L}(\lambda)$, so it follows that $U_\cP(\lambda + X) \subset U_\cP \lambda + X.$  Since $|U_\cP(\lambda +X)| = |U_\cP\lambda| = |U_\cP\lambda + X|$ as $\lambda$ and $\lambda+X$ belong to the same left $U_n$-orbit, we must have $U_\cP(\lambda + X) = U_\cP\lambda + X$, as desired.   
\end{proof}

Classifying these one-sided orbits in some sense solves the analogous two-sided problem, since for any $\lambda \in \fkn_\cP$, we have a natural surjection 
\be\label{natural-surj} \barr {cccr} U_\cP \lambda \times \lambda U_\cP  & \rightarrow & U_\cP \lambda U_\cP \\ 
 (g\lambda, \lambda h )& \mapsto &g\lambda h& \qquad\text{for $g,h \in U_\cP$.}\earr\ee  Thus, knowing the one-sided orbits gives a way of constructing all the two-sided orbits, although not uniquely.
 
   In the special case that $\lambda \in \sP_\cP$ and $\cP\vartriangleleft [[n]]$, we can explicitly describe the map (\ref{natural-surj}) without reference to a choice of elements $g,h$ in the following way.  Given $\lambda \in \sP_\cP$, define a $\FF_q$-bilinear product $*_\lambda : \fkn_\cP \times \fkn_\cP \rightarrow \fkn_\cP$ by $ (X,Y) \mapsto X *_\lambda Y$, where 
\[ \label{*-lambda} (X*_\lambda Y)_{il} = X_{il} + Y_{il} + \sum_{\substack{i < j <k < l \\ (j,k) \in \supp(\lambda)}} X_{ik} Y_{jl} (\lambda_{jk})^{-1},\qquad\text{for }X,Y\in \fkn_\cP.\] This product characterizes the map (\ref{natural-surj}) by the following lemma:

\begin{lemma}  Assume $\cP\vartriangleleft [[n]]$ and $\lambda \in \sP_\cP$.  If $X = g\lambda -\lambda$ and $Y = \lambda h - \lambda$ for $g,h \in  U_n$, then $X*_\lambda Y = g\lambda h - \lambda$.  
\end{lemma}

\begin{proof}
Since $g\lambda h - \lambda = X + Y +(g-1)\lambda(h-1)$ we have
\be \label{glh} (g\lambda h - \lambda)_{il} = X_{il} + Y_{il} + \sum_{\substack{i < j <k < l \\ (j,k) \in \supp(\lambda)}} g_{ij}\lambda_{jk}h_{kl}.\ee The set partition $\lambda$ has at most one nonzero entry in each row and column; therefore, if $(j,k)\in \supp(\lambda)$ then $\lambda_{j'k} = \lambda_{jk'} = 0$ for all $j'\neq j$ and $k'\neq k$, and so 
$ g_{ij}\lambda_{jk} = \sum_{i<j'<k} g_{ij'}\lambda_{j'k} = X_{ik}$ and 
$\lambda_{jk} h_{kl} = \sum_{j<k'<l} \lambda_{jk'} h_{k'l} = Y_{jl}.$
Hence if $(j,k)\in \supp(\lambda)$ then $g_{ij}\lambda_{jk}h_{kl} = (g_{ij}\lambda_{jk})(\lambda_{jk}h_{kl})(\lambda_{jk})^{-1} = X_{ik}Y_{jl}(\lambda_{jk})^{-1}$, and after  substituting this into (\ref{glh}), we obtain $g\lambda h -\lambda = X*_\lambda Y$.
\end{proof}

\begin{example}  Suppose $\cP\vartriangleleft[[7]]$ and $\lambda \in \sP_\cP$ are as in Example \ref{superclass-example}.  By Lemma \ref{one-sided-orbits}, the sets \[\{ \lambda +  r e_{15} + s e_{27} + t e_{36} : r,s,t \in \FF_q\}\quad\text{and}\quad \{ \lambda +  u e_{15} + v e_{26} + w e_{47} : u,v,w \in \FF_q\}\] give representatives of the distinct left and right $U_\cP$-orbits in $U_n\lambda$ and  $\lambda U_n$, respectively.  If $X =r e_{15} + s e_{27} + t e_{36}$ and $Y =u e_{15} + v e_{26} + w e_{47}$ for some $r,s,t,u,v,w\in \FF_q$, then 
\[ X*_\lambda Y = (r+u)e_{15} + s e_{27} + te_{36} + ve_{26} + we_{47} + \tfrac{rv}{c} e_{16} + \tfrac{tw}{b}e_{37}.\]
 \end{example}

Fix $\lambda \in \sP_\cP$.  By Lemma \ref{one-sided-orbits}, there are $\frac{|U_n \lambda U_n|}{|U_\cP \lambda U_\cP|} = q^{|\adj_{[[n]]}(\lambda)| - |\adj_\cP(\lambda)|} = q^{|\aux_\cP(\lambda)|}$  two-sided $U_\cP$-orbits contained in $ U_n\lambda  U_n$.  Thus, simply by order considerations we know that these $U_\cP$-orbits are in bijection with the set of $X \in \fkn_\cP$ with $\supp(X)\subset \aux_\cP(\lambda)$.  Our problem is to assign each such $X$ to a representative of a distinct two-sided $U_\cP$-orbit.  There is not really a canonical way of doing this, but using the product $*_\lambda$ defined above, we can describe one relatively natural method as follows.

\newcommand{\fkr}{\mathfrak{R}}

Given $X \in \fkn_\cP$ with $\supp(X) \subset \aux_\cP(\lambda)$, let $X_{\mathrm L}$ and $X_{\rml}$ be the unique elements of $\fkn_\cP$ such that 
\[ X = X_{\mathrm L} + X_{\rml},\qquad \supp(X_{\mathrm L}) \subset \aux_\cP^{\mathrm L}(\lambda),\qquad 
 \supp(X_{ \rml}) \subset \aux_\cP^{\mathrm R}(\lambda) - \aux_\cP^{\mathrm L}(\lambda).\]  Now define a map $\fkr_\lambda : \{ X \in \fkn_\cP : \supp(X) \subset \aux_\cP(\lambda) \} \rightarrow  U_n\lambda  U_n - \lambda$ by
 \[ \label{fkr} \fkr_\lambda (X) = X_\mathrm{L} *_\lambda X_{\rml},\qquad\text{for }X\in \fkn_\cP\text{ with }\supp(X) \subset \aux_\cP(\lambda).\]  By Lemma \ref{one-sided-orbits}, $X_{\mathrm L} \in  U_n \lambda - \lambda$ and $X_{\rml} \in \lambda  U_n -\lambda$, so $\fkr_\lambda$ maps $X$ to an element of $ U_n\lambda  U_n -\lambda$.  This definition gives us all we need to index the superclasses of $U_\cP$.

\begin{theorem}\label{normal-reps} Fix a poset $\cP\vartriangleleft [[n]]$.  Given $\lambda \in \sP_\cP$ and $X \in \fkn_\cP$ with $ \supp(X) \subset \aux_\cP(\lambda)$, define 
\[ \cK_\cP^{(\lambda,X)} \overset{\mathrm{def}}= \text{the superclass of $U_\cP$ containing $1+\lambda + \fkr_\lambda (X) \in U_\cP$}.\]
\begin{enumerate}
\item[(a)] The following map is  a bijection:
\[\label{superclasses-defn0} \barr{ccc}  
\Bigl\{
(\lambda, X) : \text{$\lambda \in \sP_\cP$ and $X \in \fkn_\cP$, $\supp(X) \subset \aux_\cP(\lambda)$}
\Bigr\} & \to & 
\Bigl\{ 
\text{Superclasses of $U_\cP$}
\Bigr\} \\
(\lambda,X) & \mapsto & \cK_\cP^{(\lambda,X)}.
 \earr\]  
\item[(b)]  Given $\lambda \in \sP_\cP$ and $X \in \fkn_\cP$ with $\supp(X) \subset \aux_\cP(\lambda)$, the superclass $\cK_\cP^{(\lambda, X)}$ has order $q^{|\adj_\cP(\lambda)|}$, which does not depend on $X$.
\end{enumerate}
\end{theorem}

\begin{remark}
A much simpler map from pairs $(\lambda,X)$ to superclasses of $U_\cP$  would  assign  $(\lambda,X)$ to the superclass containing the element $1+\lambda+X \in U_\cP$.  This map fails to be a bijection, however.  For a counterexample, take $\cP\vartriangleleft[[7]]$ to be as in Example \ref{superclass-example} and define $\mu,\nu \in \sP_\cP$ by $\mu = 1|2|3|4\larc{a}6|5|7$ and $\nu = 1\larc{b}7|2|3|4\larc{a}6|5$ for some $a,b \in \FF_q^\times$.  One checks that $\aux_\cP(\mu) = \aux_\cP(\nu)= \{ (3,6),(4,7)\}$, yet if $X = e_{36}+e_{47} \in \fkn_\cP$ then $1+\mu+X$ and $1+\nu+X$ belong to the same superclass of $U_\cP$.
 \end{remark}

Before continuing, it is helpful to consider an example illustrating our notation.

\begin{example}
Again suppose $\cP\vartriangleleft[[7]]$ and $\lambda \in \sP_\cP$ are as in Example \ref{superclass-example}.  If $X = r e_{15} + s e_{36} + t e_{47} \in \fkn_\cP$ for some $r,s,t \in \FF_q$, then 
\[\fkr_\lambda(X) = (re_{15} + se_{36})*_\lambda(te_{47}) = X + b^{-1}st e_{37}\] and $\cK_\cP^{(\lambda,X)}$ is the superclass of $U_\cP$ containing the element 
 \[ 1+\lambda+\fkr_\lambda(X) = \(\begin{array}{ccccccc} 
1 & 0 & 0 & a & r & 0 & 0 \\
 & 1 & 0 & 0 & c & 0 & 0 \\
 &  & 1 & 0 & 0& s & d + b^{-1}st \\
 &  &  & 1 & 0 & b & t \\
 &  &  &  & 1 & 0 & 0 \\
 &  &  &  &  & 1 & 0 \\
 &  &  &  &  &  & 1 \end{array}\) \in U_\cP.\]
\end{example}

The proof of Theorem \ref{normal-reps} depends on the following somewhat technical lemma.

\begin{lemma} \label{technical-lemma} Fix a poset $\cP\vartriangleleft [[n]]$ and $\lambda \in \sP_\cP$, and let $X_i \in \fkn_\cP$ have $\supp(X_i)\subset \aux_\cP(\lambda) - \aux_\cP^{\mathrm L}(\lambda)$ for $i=1,2$.  Then there are elements $h_1,h_2 \in U_n$ with $X_i = \lambda h_i - \lambda$  such that $\supp(\lambda h_1h_2^{-1}) \cap \aux_\cP^{\mathrm L}(\lambda) = \varnothing$.

\end{lemma}

\begin{proof}
Enumerate the positions in $\aux_\cP^{}(\lambda) - \aux_\cP^{\mathrm L}(\lambda)$ as $(i_1,k_1),\dots,(i_r,k_r)$ such that $k_1 \geq \dots \geq k_r$.  For each index $t=1,\dots,r$, let $j_t$ be the column with $(i_t,j_t) \in \supp(\lambda)$.  We claim that the lemma holds if we take $h_i \in  U_n$ to be the element 
\[\label{product-elem} h_i = 1 + \sum_{t=1}^r  (X_i)_{i_tk_t} (\lambda_{i_tj_t})^{-1} e_{j_tk_t} = \prod_{t=1}^{r} \(1 + (X_i)_{i_tk_t} (\lambda_{i_tj_t})^{-1} e_{j_tk_t}\),\] where the factors in the product are multiplied in order from left to right.  (In other words, we evaluate the expression $\prod_{t=1}^r x_t$ as $ x_1x_2\cdots x_r$.)

To show this, we first note that $(j_t,k_t) \in [[n]] - \cP$ for all $t$ by the definitions given in (\ref{adj}) so our elements $h_i \in  U_n$ are well-defined.  We next observe that since $k_1\geq \dots \geq k_r$, we have $e_{j_tk_t} e_{j_{t+1}k_{t+1}} = 0$ for all $t$, and so the given sum and product formulas for $h_i$ are equal.
Since $\lambda e_{j_tk_t} = \lambda_{i_tj_t} e_{i_tk_t}$ by construction, it follows immediately that $\lambda h_i - \lambda = X_i$. 

To prove that $\supp(\lambda h_1h_2^{-1}) \cap \aux_\cP^{\mathrm L}(\lambda) = \varnothing$, consider the subspace $\cS\subset \fkn_\cP$ of matrices whose nonzero positions coincide with or lie below nonzero positions of $\lambda$:
\[ \cS = \{  Y \in \fkn_\cP: Y_{ik} \neq 0 \Rightarrow \lambda_{jk} =0\text{ for all }j> i \}.\]  Clearly $\lambda \in \cS$.  Now observe that for any $t$, $(i_t,k_t) \in \aux_\cP(\lambda) - \aux_\cP^{\mathrm{L}}(\lambda)$ implies that $(i_t,k_t) \notin \adj_{[[n]]}^{\mathrm L}(\lambda)$, and so necessarily $(j,k_t) \notin \supp(\lambda)$ for all $j > i_t$.  It follows from this that both $X_1,X_2 \in \cS$.  Furthermore, if $Y \in \cS$, then $Y' = Ye_{j_tk_t} \in \cS$.  This follows since by construction $(i_t,j_t) \in \supp(\lambda)$, so the nonzero positions of $Y$ in the $j_t$th colum all lie below the $i_t$th row; hence $Y'$ has nonzero positions only in the $k_t$th column below the $i_t$th row, and so $Y' \in \cS$ since $(j,k_t) \notin \supp(\lambda)$ for all $j > i_t$.

It follows immediately that we have $\cS h_2^{-1} \subset \cS$, since if $Y\in \cS$ then $Y (1+ce_{j_tk_t})^{-1} = Y(1-ce_{j_tk_t}) = Y - cY e_{j_tk_t} \in \cS$ for all $t=1,\dots,r$ and $c\in \FF_q$.  Thus, in particular $X_1 h_2^{-1} \in \cS$ and $\lambda h_2^{-1} \in \cS$, so $\lambda h_1h_2^{-1} = X_1h_2^{-1} + \lambda h_2^{-1} \in \cS$.  But $Y \in \cS$ implies that $\supp(Y) \cap \adj_{[[n]]}^{\mathrm L}(\lambda) = \varnothing$, so we certainly have $\supp(\lambda h_1h_2^{-1}) \cap \aux_\cP^{\mathrm L}(\lambda) = \varnothing$.
\end{proof}

We may now prove the theorem.

\begin{proof}[Proof of Theorem \ref{normal-reps}]
In light of the two-sided action of $ U_n$ on $\fkn_\cP$, we know that every superclass of $U_\cP$ is a subset of a superclass $ U_n \lambda  U_n$ of $ U_n$ for some $\lambda \in \sP_\cP$, and that conversely, each superclass $ U_n \lambda  U_n$ of $ U_n$ decomposes as a disjoint union of superclasses of $U_\cP$ of equal order.

Fix $\lambda \in \sP_\cP$.  Then by Lemma \ref{one-sided-orbits}, we know that there are $\frac{ | U_n \lambda  U_n|}{|U_\cP \lambda U_\cP|} = q^{|\adj_{[[n]]}(\lambda)|- |\adj_\cP(\lambda)|}= q^{|\aux_\cP(\lambda)|}$ distinct two-sided $U_\cP$-orbits in $ U_n\lambda  U_n$.  Hence to show that the given map $(\lambda, X) \mapsto \cK_\cP^g$ is a bijection, it suffices to prove that when $\lambda \in \sP_\cP$ is fixed, the $q^{|\aux_\cP(\lambda)|}$ elements of the form $\lambda + \fkr_\lambda(X)$ where $\supp(X) \subset \aux_\cP(\lambda)$ belong to distinct two-sided $U_\cP$-orbits.  

To this end, suppose the contrary: that for some $X_i \in \fkn_\cP$ with $\supp(X_i) \subset \aux_\cP(\lambda)$ for $i=1,2$, we have $x(\lambda + \fkr_\lambda(X_1))y = \lambda + \fkr_\lambda(X_2)$ for some $x,y \in U_\cP$.  Write $(X_i)_{\mathrm L} = g_i \lambda-\lambda$ and $(X_i)_{\rml} = \lambda h_i-\lambda$ for $g_i,h_i \in  U_n$, where $h_i$ is taken to be as in Lemma \ref{technical-lemma}.  Then $\lambda +\fkr_\lambda(X_i) = g_i\lambda h_i$ and we have $\lambda h_1h_2^{-1} y' = x'g_1^{-1}g_2\lambda$, where $x' = g_1^{-1}x^{-1} g_1 \in U_{\cP}$ and $y' = h_2yh_2^{-1} \in U_\cP$ by normality.

By Lemma \ref{one-sided-orbits}, we have $g_1^{-1}g_2 \lambda \in U_\cP(\lambda +X)$ for some $X \in \fkn_\cP$ with $\supp(X) \subset \aux_\cP^{\mathrm L}(\lambda)$.  Similarly, by Lemmas \ref{one-sided-orbits} and \ref{technical-lemma} we have $\lambda h_1h_2^{-1} \in (\lambda + Y)U_\cP$ for some $Y \in \fkn_\cP$ with $\supp(Y) \subset \aux_\cP^{\mathrm R}(\lambda) - \aux_\cP^{\mathrm L}(\lambda)$.  Putting this together, we have \[\lambda h_1h_2^{-1} y' = x'g_1^{-1}g_2\lambda \in U_\cP(\lambda + X) \cap (\lambda + Y)U_\cP.\]  But $\supp(X)\cap \supp(Y) = \varnothing$, so by the characterization of the left and right $U_\cP$-orbits in Lemma \ref{one-sided-orbits}, it necessarily follows that $X=Y = 0$.  Thus, $U_\cP(\lambda + (X_1)_{\mathrm L})  = U_\cP(\lambda +(X_2)_{\mathrm L})$ and $(\lambda + (X_1)_{\rml}) U_\cP = (\lambda +(X_2)_{\rml})U_\cP$, so $(X_1)_{\mathrm L} = (X_2)_{\mathrm L}$ and $(X_1)_{\rml} = (X_2)_{\rml}$ by Lemma \ref{one-sided-orbits}, and consequently $X_1 = X_2$.

This proves that the map $(\lambda,X) \mapsto \cK^g_\cP$ is a bijection.   The last part of the theorem concerning the sizes of superclasses follows directly from Lemma \ref{one-sided-orbits}.\end{proof}

\section{Supercharacter Constructions}\label{supercharacter-section}

Fix a poset $\cP\vartriangleleft [[n]]$.  The constructions and arguments needed to classify the supercharacters of $U_\cP$ closely mirror those of the previous section.  We begin by defining several sets of positions which serve the same purpose as, and are dual to, the sets $\adj$ and $\aux$ above.  Given a set partition $\lambda \in \sP^*_{\cP}$ and a poset $\cQ \in \{ \cP, [[n]]\}$,  
define
\be \label{coadj} \ba \coadj_\cQ^{\mathrm{L}}(\lambda) &= \{ (j,k) \in \cP : \exists\ (i,k) \in \supp(\lambda)\text{ with }(i,j) \in \cQ \},\\
\coadj_\cQ^{\mathrm{R}}(\lambda) &= \{ (i,j) \in \cP : \exists\ (i,k) \in \supp(\lambda)\text{ with }(j,k) \in \cQ \}, \\
\coadj_\cQ(\lambda) &= \coadj_\cQ^{\mathrm{L}}(\lambda) \cup \coadj_\cQ^{\mathrm{R}}(\lambda).
\ea \ee
Note the dependence on $\cP$  in the definition of $\coadj_\cQ^{\mathrm L}$ and $\coadj_\cQ^{\mathrm R}$.  Since we distinguish linear functionals by their domains, and so view $\sP_\cQ^*$ and $\sP_{\cR}^*$ as disjoint sets when $\cQ$ and $\cR$ are distinct posets, this slight abuse of notation does not present any significant ambiguity.  Next define  
\be\label{coaux}\ba \coaux^{\mathrm{L}}_\cQ(\lambda) &= \coadj_{[[n]]}^{\mathrm L} (\lambda) -  \coadj_{\cQ}^{\mathrm{L}}(\lambda),\\
\coaux^{\mathrm{R}}_\cQ(\lambda) &= \coadj_{[[n]]}^{\mathrm R} (\lambda) -  \coadj_{\cQ}^{\mathrm{R}}(\lambda),\\
\coaux_\cQ(\lambda) &= \coadj_{[[n]]}^{} (\lambda) -  \coadj_{\cQ}^{}(\lambda).
\ea
\ee 
These sets of positions are also always subsets of $\cP$, and in particular, they are empty when $\cQ = [[n]]$.  

The sets in both (\ref{coadj}) and (\ref{coaux}) are disjoint from $\supp(\lambda)$.  As in the preceding section, this follows since $\supp(\lambda)$ contains at most one position in each row and column as  $\lambda$ is a set partition, and the positions in (\ref{coadj}) and (\ref{coaux}) each lie either in the same row and strictly to the left of a position in $\supp(\lambda)$, or in the same column and strictly below a position in $\supp(\lambda)$.

\begin{example} \label{supercharacter-example} Suppose $\cP\vartriangleleft [[7]]$ is given as in Example \ref{superclass-example}.   If $\lambda \in \sP_\cP^*$ is the set partition
\[\lambda =1\larc{a}4\larc{b}6|2\larc{c}7|3\larc{d}5= 
\(\begin{array}{ccccccc} 
0 & 0 & 0 & a & 0 & 0 & 0 \\
 & 0 & 0 & 0 & 0 & 0 & c \\
 &  & 0 & 0 & d & 0 & 0 \\
 &  &  & 0 & 0 & b & 0 \\
 &  &  &  & 0 & 0 & 0 \\
 &  &  &  &  & 0 & 0 \\
 &  &  &  &  &  & 0 \end{array}\) \in \fkn_\cP^* \]
then we compute \[ \ba 
&\coadj_{[[7]]}^{\mathrm{L}}(\lambda) = \{ (2,4), (3,7),(4,7),(5,7)\}, \\ 
&\coadj_\cP^{\mathrm{L}}(\lambda) = \{ (4,7),(5,7) \},\\ 
&\coaux_\cP^{\mathrm{L}}(\lambda) = \{ (2,4), (3,7)\},
\ea\qquad\text{and}\qquad\ba
&\coadj_{[[7]]}^{\mathrm{R}}(\lambda) = \{ (1,3), (2,4),(2,5),(2,6)\}, \\ 
&\coadj_\cP^{\mathrm{R}}(\lambda) = \{ (2,4),(2,5) \},\\ 
&\coaux_\cP^{\mathrm{R}}(\lambda) = \{ (1,3), (2,6)\} .
\ea\]  
Thus, $\coadj_{[[7]]}(\lambda)$, $\coadj_{\cP}(\lambda)$, and $\coaux_\cP(\lambda)$  are the following sets of positions in a $7\times 7$ matrix:
\[ \barr{c} \(\begin{array}{ccccccc} 
0 & 0 & {\color{Green}\blacksquare} & {\color{Gray}a} & 0 & 0  & 0  \\
 & 0 & 0& {\color{Green}\blacksquare} & {\color{Green}\blacksquare} & {\color{Green}\blacksquare}  & {\color{Gray}{c}}  \\
 &  & 0 & 0 &  {\color{Gray}d}& 0  & {\color{Green}\blacksquare}  \\
 &  &  & 0 & 0 & {\color{Gray}b} & {\color{Green}\blacksquare}  \\
 &  &  &  & 0 & 0 & {\color{Green}\blacksquare} \\
 &  &  &  &  & 0 & 0 \\
 &  &  &  &  &  & 0 \end{array}\), \\
 {\color{Green}\blacksquare} \in \coadj_{[[n]]}(\lambda)
 \earr
 \quad  \barr{c}\(\begin{array}{ccccccc} 
0 & 0& 0& {\color{Gray}a} & 0 & 0  & 0  \\
 & 0 &0 & {\color{Blue}\blacksquare} & {\color{Blue}\blacksquare} & 0  & {\color{Gray}{c}}  \\
 &  & 0 & 0 &  {\color{Gray}d}& 0  & 0  \\
 &  &  & 0 & 0& {\color{Gray}b} & {\color{Blue}\blacksquare}  \\
 &  &  &  & 0 & 0& {\color{Blue}\blacksquare} \\
 &  &  &  &  & 0 & 0\\
 &  &  &  &  &  & 0 \end{array}\),\\
 {\color{Blue}\blacksquare} \in \coadj_\cP(\lambda)
 \earr
 \quad \barr{c}
  \(\begin{array}{ccccccc} 
0 & 0 & {\color{Red}\blacksquare} & {\color{Gray}a} & 0 & 0  & 0  \\
 & 0 & 0 & 0& 0& {\color{Red}\blacksquare}  & {\color{Gray}{c}}  \\
 &  & 0 & 0&  {\color{Gray}d}& 0  & {\color{Red}\blacksquare}  \\
 &  &  & 0 & 0 & {\color{Gray}b} & 0  \\
 &  &  &  & 0 & 0 & 0 \\
 &  &  &  &  & 0 & 0 \\
 &  &  &  &  &  & 0 \end{array}\). \\ 
 {\color{Red}\blacksquare} \in \coaux_\cP(\lambda)
 \earr\] 
 \end{example}
 
As in Section \ref{superclass-section},  our primary goal is to construct a bijection
\[ \barr{ccc}  
\Bigl\{ 
\text{Supercharacters of $U_\cP$}
\Bigr\}
 & \leftrightarrow & 
 \Bigl\{
(\lambda, X) :  \text{$\lambda \in \sP_\cP^*$ and $\eta \in \fkn_\cP^*$, $\supp(\eta) \subset \coaux_\cP(\lambda)$} 
\Bigr\}.
\earr\]
For this, we first classify the one-sided $U_\cP$-orbits in $\fkn_\cP^*$ with the following lemma.

\begin{lemma} \label{one-sided-coorbits} Fix a poset $\cP\vartriangleleft [[n]]$ and a set partition $\lambda \in \sP_\cP^*$, and let $\cQ \in \{\cP, [[n]]\}$.  
\begin{enumerate}
\item[(a)] Each left $U_\cQ$-orbit in $U_n\lambda$ has the form
 $\left\{  \lambda +\eta +\nu \in \fkn_\cP^*: \supp(\nu) \subset \coadj_\cQ^{\mathrm L}(\lambda)\right\} $ for a unique functional $\eta \in \fkn_\cP^*$ with $\supp(\eta)\subset \coaux_\cQ^{\mathrm L}(\lambda)$.
 
 \item[(b)] Each right $U_\cQ$-orbit in $\lambda U_n$ is of the form
 $  \left\{  \lambda +\mu +\nu  \in \fkn_\cP : \supp(\nu) \subset \coadj_\cQ^{\mathrm R}(\lambda)\right\}
$ for a unique functional $\mu \in \fkn_\cP^*$ with $\supp(\mu) \subset \coaux_\cQ^{\mathrm R}(\lambda)$.  
\item[(c)] Consequently $|U_\cQ \lambda U_\cQ| = q^{|\coadj_\cQ(\lambda)|}$.
\end{enumerate}

\end{lemma}

\begin{proof}
After one verifies for $g \in U_n$ the identities $(g^{-1}\lambda - \lambda)_{jk} =  \sum_{i<j} g_{ij} \lambda_{ik}$ when $(j,k) \in \cP$ and $(\lambda g^{-1}-\lambda)_{ij} = \sum_{j<k} g_{jk}\lambda_{ik}$ when $(i,j) \in \cP$
 (which are given as Theorem 4.2 in \cite{DT09}  when $g \in U_\cP$), the proof of this lemma follows exactly the same structure as the proof of Lemma \ref{one-sided-orbits}.
\end{proof}

For any $\lambda \in \fkn_\cP^*$, we  again have a natural surjection 
\be\label{co-natural-surj} \barr {cccr} U_\cP \lambda \times \lambda U_\cP  & \rightarrow & U_\cP \lambda U_\cP \\ 
 (g\lambda, \lambda h )& \mapsto &g\lambda h & \qquad\text{for $g,h \in U_\cP$.}\earr\ee  Once again, when $\lambda \in \sP_\cP^*$ and $\cP\vartriangleleft [[n]]$, we can explicitly describe this map without reference to the elements $g,h$ in the following way.  Given $\lambda \in \sP_\cP^*$, define a $\FF_q$-bilinear product $*_\lambda : \fkn_\cP^* \times \fkn_\cP^* \rightarrow \fkn_\cP^*$ (with slight abuse of notation using the same symbol as above) by $ (\eta,\mu) \mapsto \eta *_\lambda \mu$, where 
\[ \label{co*-lambda}  (\eta *_\lambda \mu)_{jk} = \eta_{jk} + \mu_{jk} + \sum_{\substack{i< j<k< l \\ (i,l)\in\supp(\lambda)}} \eta_{jl}\mu_{ik} (\lambda_{il})^{-1},\qquad\text{for }\eta,\mu\in \fkn_\cP^*\text{ and }(j,k) \in \cP.\] This product characterizes the map (\ref{co-natural-surj}) by the following lemma:

\begin{lemma}  Assume $\cP\vartriangleleft [[n]]$ and $\lambda \in \sP_\cP^*$.  If $\eta = g\lambda -\lambda$ and $\mu = \lambda h - \lambda$ for $g,h \in  U_n$, then $\eta*_\lambda \mu = g\lambda h - \lambda$.  
\end{lemma}

\begin{proof}
Replace $g,h$ with their inverses, so that $\eta = g^{-1}\lambda - \lambda$ and $\mu = \lambda h^{-1}-\lambda$.  One can then check that
\be \label{glh*} (g^{-1}\lambda h^{-1} - \lambda)_{jk} = \eta_{jk} + \mu_{jk} + \sum_{\substack{i < j <k < l \\ (j,k) \in \supp(\lambda)}} g_{ij}\lambda_{il}h_{kl},\qquad\text{for }(j,k) \in \cP.\ee The set partition $\lambda$ has at most one nonzero entry in each row and column; therefore, if $(i,l)\in \supp(\lambda)$ then $\lambda_{i'l} = \lambda_{il'} = 0$ for all $i'\neq i$ and $l'\neq l$, and so 
$g_{ij}\lambda_{il} = \sum_{i'< j} g_{i'j}\lambda_{i'l} = \eta_{jl}$ and 
$\lambda_{il} h_{kl} = \sum_{k<l'} \lambda_{il'} h_{kl'} = \mu_{ik}.$
Hence if $(i,l)\in \supp(\lambda)$ then $g_{ij}\lambda_{il}h_{kl} = (g_{ij}\lambda_{il})(\lambda_{il}h_{kl})(\lambda_{il})^{-1} = \eta_{jl}\mu_{ik}(\lambda_{il})^{-1}$, and after  substituting this into (\ref{glh*}), we obtain $g\lambda h -\lambda = \eta*_\lambda \mu$.
\end{proof}

\begin{example}  Suppose $\cP\vartriangleleft[[7]]$ and $\lambda \in \sP_\cP^*$ are as in Example \ref{supercharacter-example}.  Write $e_{ij}^* \in \fkn_\cP^*$ to denote the linear functional with $e_{ij}^*(X) = X_{ij}$.  Then by Lemma \ref{one-sided-coorbits}, the sets \[\{ \lambda +  r e_{24}^* + s e_{57}^*  : r,s \in \FF_q\}\quad\text{and}\quad \{ \lambda +  t e_{13}^* + u e_{26}^*  : t,u \in \FF_q\}\] give representatives of the distinct left and right $U_\cP$-orbits in $U_n\lambda$ and  $\lambda U_n$, respectively.  If $\eta = r e_{24}^* + s e_{37}^*$ and $\mu =t e_{13}^* + u e_{26}^*$ for some $r,s,t,u\in \FF_q$, then 
\[ \eta*_\lambda \mu = r e_{24}^* + s e_{37}^* + t e_{13}^* + u e_{26}^* + c^{-1} sue_{36}^* .\]
 \end{example}

Fix $\lambda \in \sP^*_\cP$.  By the same order considerations as in the superclass case, we know that the set of two-sided $U_\cP$-orbits contained in $ U_n\lambda  U_n$ is in bijection with the set of $\eta \in \fkn_\cP^*$ with $\supp(\eta)\subset \coaux_\cP(\lambda)$.  The difficulty is again to assign each such $\eta$ to a representative of a distinct two-sided $U_\cP$-orbit.  Mirroring the superclass case, we have the following (non-canonical) construction.

Given $\eta \in \fkn_\cP^*$ with $\supp(\eta) \subset \coaux_\cP(\lambda)$, let $\eta_{\mathrm L}$ and $\eta_{\rml}$ be the unique elements of $\fkn_\cP^*$ such that 
\[ \eta = \eta_{\mathrm L} + \eta_{\rml},\qquad \supp(\eta_{\mathrm L}) \subset \coaux_\cP^{\mathrm L}(\lambda),\qquad 
 \supp(\eta_{\rml}) \subset \coaux_\cP^{\mathrm R}(\lambda) - \coaux_\cP^{\mathrm L}(\lambda).\]  Now define a map $\fkr^*_\lambda : \{ \eta \in \fkn_\cP^* : \supp(\eta) \subset \coaux_\cP(\lambda) \} \rightarrow  U_n\lambda  U_n - \lambda$ by
 \[ \label{fkr*} \fkr^*_\lambda (\eta) = \eta_\mathrm{L} *_\lambda \eta_{\rml},\qquad\text{for }\eta\in \fkn_\cP\text{ with }\supp(\eta) \subset \coaux_\cP(\lambda).\]  By Lemma \ref{one-sided-coorbits}, $\eta_{\mathrm L} \in  U_n \lambda - \lambda$ and $\eta_{\rml} \in \lambda  U_n -\lambda$, so $\fkr^*_\lambda$ maps $\eta$ to an element of $ U_n\lambda  U_n -\lambda$.  We can now classify the supercharacters of $U_\cP$.

\begin{theorem}\label{co-normal-reps} Fix a poset $\cP\vartriangleleft [[n]]$.  Given $\lambda \in \sP_\cP^*$ and $\eta \in \fkn_\cP^*$ with $\supp(\eta) \subset \coaux_\cP(\lambda)$, define 
\[ \chi_\cP^{(\lambda,\eta)} \overset{\mathrm{def}}= \text{ the supercharacter indexed by $\lambda + \fkr^*_\lambda(\eta) \in \fkn_\cP^*$}.\]
\begin{enumerate}
\item[(a)] The following map is  a bijection:
\[\label{supercharacters-defn0} \barr{ccc}  
\Bigl\{
(\lambda, \eta) :  \text{$\lambda \in \sP_\cP^*$ and $\eta \in \fkn_\cP^*$, $\supp(\eta) \subset \coaux_\cP(\lambda)$} 
\Bigr\} & \to & 
\Bigl\{ 
\text{Supercharacters of $U_\cP$}
\Bigr\} \\
(\lambda, \eta) &\mapsto &\chi_\cP^{(\lambda,\eta)}. 
\earr\] 
\item[(b)]   For $\lambda \in \sP_\cP^*$ and $\eta \in \fkn_\cP^*$ with $\supp(\eta) \subset \coaux_\cP(\lambda)$, we have
\[ \chi_\cP^{(\lambda,\eta)}(1) = q^{|\coadj_\cP^{\mathrm{L}}(\lambda)|} = q^{|\coadj_\cP^{\mathrm{R}}(\lambda)|}\qquad\text{and}\qquad \left\langle \chi_\cP^{(\lambda,\eta)}, \chi_\cP^{(\lambda,\eta)} \right \rangle_{U_\cP} = q^{|\coadj_\cP^{\mathrm L}(\lambda) \cap \coadj_\cP^{\mathrm R}(\lambda)|},\]
 both of which do not depend on $\eta$.
 \end{enumerate}
 \end{theorem}

\begin{remark}
A much simpler map from pairs $(\lambda,\eta)$ to supercharacters of $U_\cP$  would assign  $(\lambda,\eta)$ to the supercharacter indexed by $\lambda+\eta \in \fkn_\cP^*$.  However, as in the superclass case, this naive map  fails to be a bijection.  One sees this by examining essentially the same counterexample as in the remark following Theorem \ref{normal-reps}.  Namley, take $\cP\vartriangleleft[[7]]$ to be as in Example \ref{superclass-example} and define $\mu,\nu \in \sP_\cP^*$ by $\mu = 1|2|3|4\larc{a}6|5|7$ and $\nu = 1\larc{b}7|2|3|4\larc{a}6|5$ for some $a,b \in \FF_q^\times$.  One checks that $\coaux_\cP(\mu) = \coaux_\cP(\nu)= \{ (1,6),(2,7)\}$, yet if $\eta = e^*_{16}+e^*_{27} \in \fkn_\cP^*$ then $\mu+\eta$ and $\nu+\eta$ belong to the same two-sided $U_\cP$ orbit and so index the same supercharacter.
 \end{remark}

Our exposition of the proof of this result is not nearly as long as in the superclass case.  The proof likewise depends on a technical lemma.

\begin{lemma} \label{co-technical-lemma} Fix a poset $\cP\vartriangleleft [[n]]$ and $\lambda \in \sP_\cP^*$, and let $\eta_i \in \fkn_\cP$  have $\supp(\eta_i)\subset \coaux_\cP(\lambda) - \coaux_\cP^{\mathrm L}(\lambda)$ for $i=1,2$.  Then there are elements $h_1,h_2 \in U_n$ with  $\eta_i = \lambda h_i - \lambda$  such that $\supp(\lambda h_1h_2^{-1}) \cap \coaux_\cP^{\mathrm L}(\lambda) = \varnothing$.

\end{lemma}

\begin{proof}
Enumerate the positions in $\coaux_\cP^{}(\lambda) - \coaux_\cP^{\mathrm L}(\lambda)$ as $(i_1,j_1),\dots,(i_r,j_r)$ such that $j_1 \leq \dots \leq j_r$.  For each index $t=1,\dots,r$, let $k_t$ be the column with $(i_t,k_t) \in \supp(\lambda)$.  We claim that the lemma holds with 
\[\label{product-elem*} h_i = \(1 + \sum_{t=1}^r  (\eta_i)_{i_tj_t} (\lambda_{i_tk_t})^{-1} e_{j_tk_t}\)^{-1} = \prod_{t=1}^{r} \(1 + (\eta_i)_{i_tj_t} (\lambda_{i_tk_t})^{-1} e_{j_tk_t}\)^{-1},\] where the factors in the product are multiplied in order from left to right.   
  Proving this is straightforward and follows almost exactly the same argument as the proof of Lemma \ref{technical-lemma}, with only minor adjustments reflecting the fact that our action is now on $\fkn_\cP^*$ rather than $\fkn_\cP$.  
\end{proof}

\begin{proof}[Proof of Theorem \ref{co-normal-reps}]
Given the preceding lemma, the proof of the theorem is essentially just a repetition of the proof of Theorem \ref{normal-reps}, where we change all instances of $\fkn_\cP$ to $\fkn_\cP^*$, and we update all references to the corresponding lemmas in this section.
\end{proof}

\begin{example} Again suppose $\cP\vartriangleleft[[7]]$ and $\lambda \in \sP_\cP$ are as in Example \ref{supercharacter-example}.  If $\eta = r e_{13}^* + s e_{26}^* + t e_{37}^* \in \fkn_\cP^*$, then 
\[ \fkr_\lambda^*(\eta) = (te_{37}^*)*_\lambda (r e_{13}^* + s e_{26}^*) = \eta + c^{-1}st e_{35}^*\] and $\chi_\cP^{(\lambda,\eta)} = \chi_\cP^\nu$ where $\nu \in \fkn_\cP^*$ is the $\FF_q$-linear functional defined by the matrix 
 \[ \nu = 
 \(\begin{array}{ccccccc} 
0 & 0 & r & a & 0 & 0 & 0 \\
 & 0 & 0 & 0 & 0 & s & c \\
 &  & 0 & 0 & d & stc^{-1} & t \\
 &  &  & 0 & 0 & b & 0 \\
 &  &  &  & 0 & 0 & 0 \\
 &  &  &  &  & 0 & 0 \\
 &  &  &  &  &  & 0 \end{array}\) \in \fkn_\cP^*.\]
\end{example}

Since a character $\chi$ is irreducible if and only if its inner product with itself is one, we have the following corollary to Theorem \ref{co-normal-reps}.

\begin{corollary}\label{irred-criterion}
Let $\cP \vartriangleleft [[n]]$, and suppose $\lambda \in \sP_\cP^*$ and $\eta \in \fkn_\cP^*$ has $\supp(\eta) \subset \coaux_\cP(\lambda)$.  Then the supercharacter $\chi_\cP^{(\lambda,\eta)}$ is irreducible if and only if   $(i,k),(j,l) \in \supp(\lambda)$ implies $\{ (i,j) ,(j,k),(k,l)\} \not\subset\cP$ whenever  $i <j<k<l$.
\end{corollary}

When $\cP =[[n]]$ and $\lambda \in \sP_n^*$, this result becomes the combinatorial condition that $\chi_{[[n]]}^{\lambda}$ is irreducible if and only if  the set partition $\lambda$ is \emph{non-crossing}.  A set partition of $[[n]]$ is non-crossing if when the numbers $1,2,\dots, n$ are arranged consecutively in a circle, none of the chords connecting $i$ and $j$ for $(i,j) \in \supp(\lambda)$ intersect inside the circle.

As another corollary, we can describe the restriction of supercharacters from $ U_n$ to the normal subgroup $U_\cP$ explicitly.

\begin{corollary} Fix a poset $\cP\vartriangleleft [[n]]$.   Choose $\lambda \in \sP_n^*$ and let $\mu \in \sP_{\cP}^*$ be the restriction of $\lambda$ to $\fkn_\cP$.  Then the restriction of $\chi^\lambda$ to $U_\cP$ decomposes as
\[ \Res_{U_\cP}^{U_n} \(\chi^\lambda \)  =  q^{ c}\sum_\eta \chi_{\cP}^{(\mu,\eta)},\qquad\text{where $c = |\coadj_{[[n]]}^{\mathrm L}(\lambda)| - |\coadj_\cP^{\mathrm L}(\mu)|-|\coaux_\cP(\mu)|$}\] and where the sum is over all $\eta \in \fkn_\cP^*$ with  $\supp(\eta) \subset \coaux_\cP(\mu)$.
\end{corollary}

\begin{proof}
This result follows by applying  Theorem 5.1 in \cite{M09}
\end{proof}

As remarked at the end of Section \ref{sect2}, each supercharacter of $U_n$ is given by a product of irreducible supercharacters.  This is not true for all algebra groups under the supercharacter theory defined in \cite{DI06}; as a counterexample, the abelian algebra group $G = \{ g \in U_3 : g_{12} = g_{23}\}$ has supercharacters with degree two.
However, using the theorem we can show that it does hold when $U_\cP \vartriangleleft U_n$.  In particular, let $\cP \vartriangleleft [[n]]$ and call a supercharacter $\chi_\cP^{(\lambda,\eta)}$   \emph{elementary}  if $|\supp(\lambda)| = 1$.  Such supercharacters are characterized explicitly as those of the form $\chi^\mu_\cP$ where $\mu \in \fkn_\cP^*$ and for some $(i,l) \in \cP$, we have 
\begin{enumerate}
\item[(1)] $\mu_{il} \neq 0$;
\item[(2)] $\mu_{jk} \neq 0$ only if $i\leq j \leq \max \{ i' : (i,i') \in [[n]] \setminus \cP \}$ and 
$l \geq k \geq \min\{ l' : (l',l) \in  [[n]] \setminus \cP \}$;
\item[(3)] $\mu_{jk} = \mu_{ik} \mu_{jl} /  \mu_{il}$ if $\mu_{jk} \neq 0$.
\end{enumerate}
By Corollary \ref{irred-criterion}, each elementary supercharacter of $U_\cP$ is irreducible.   Observe  also by Lemma \ref{one-sided-coorbits} that  if $\chi^\mu_\cP$ is elementary, then $\chi_\cP^{g\mu h}$ is elementary for all $g,h \in U_n$.  
We now have the following result.

\begin{proposition}
If $\cP \vartriangleleft [[n]]$ then each supercharacter of $U_\cP$ is a product of (irreducible) elementary supercharacters.
\end{proposition}

\begin{proof}
Choose an arbitrary supercharacter of $U_\cP$; by our classification theorem, this character is given by $\chi_\cP^\mu$ for a linear functional $\mu \in \fkn_\cP^*$ with $\mu = g\lambda h$ for some $g,h \in U_n$ and $\lambda \in \sP_\cP^*$.  Write $\lambda = \alpha_1 + \dots +\alpha_k$ for some $\alpha_i \in \fkn_\cP^*$ with $|\supp(\alpha_i)|=1$; then the positions in $\supp(\alpha_i)$ all lie in distinct row and columns.   We claim that 
$\chi_\cP^\lambda = \chi_\cP^{\alpha_1} \cdots \chi_\cP^{\alpha_k}.$  This follows most readily from module considerations.  One easily checks that the map 
\[ \barr{ccc} U_\cP \alpha_1 \times \cdots \times U_\cP \alpha_k & \to & U_\cP \lambda \\
 (\nu_1,\dots,\nu_k) & \mapsto & \nu_1 + \dots + \nu_k \earr\] is a bijection, and it follows immediately that the linear map defined on basis elements by 
 \[ \barr{ccc} V^{\alpha_1}_\cP \otimes \cdots \otimes V^{\alpha_k}_\cP & \to & V^\lambda_\cP \\
 v_{\nu_1} \otimes \cdots \otimes v_{\nu_k} & \mapsto & v_{\nu_1 + \dots + \nu_k} \earr\]  is a $U_\cP$-module isomorphism, which establishes our factorization of $\chi_\cP^\lambda$ as a product of the elementary supercharacters $\chi_\cP^{\alpha_i}$.  It then follows from the formula (\ref{formula}) that $\chi_\cP^\mu = \chi_\cP^{g\lambda h} = \chi_\cP^{g \alpha_1 h} \cdots \chi_\cP^{g \alpha_k h}$ and so $\chi_\cP^\mu$ is also a product of elementary supercharacters.
\end{proof}

\section{$\FF_q$-labeled Posets and the Supercharacters of $U_\cP\vartriangleleft U_n$}\label{posets-section}

As promised in the introduction, we now carry the classification of the supercharacters of $U_\cP$ a step further, by providing a combinatorial interpretation of Theorem \ref{co-normal-reps}.  To this end, we first    show how one can naturally represent the supercharacters $\chi_\cP^{(\lambda,\eta)}$ as posets labeled by elements of $\FF_q$.  Then, going in the opposite direction, we classify the set of such representative posets by a  graph theoretic condition involving only $\cP$.

To begin this program, we must first make a few definitions.    Recall that if $\cP$ is a poset, then $\cP^\cov$ is the set of its covers; i.e., the elements $(i,k) \in \cP$ such that there is no $j$ with $(i,j),(j,k) \in \cP$.  
Analogous to an $\FF_q$-labeled set partition, we define an \emph{$\FF_q$-labeled poset} to be a poset $\cP$ with a map $\cP^\cov \rightarrow \FF_q^\times$ which labels each cover with a nonzero element of $\FF_q$.      We can think of this labeling as assigning an element of $\FF_q^\times$ to each edge in the Hasse diagram of $\cP$.  For each $(i,j) \in \cP^\cov$, we let $\cP_{ij}$ denote the corresponding label, and for each $(i,j) \notin \cP^\cov$ we set $\cP_{ij} = 0$.  With minor abuse of notation, we refer to a labeled poset by just its poset structure $\cP$.  In particular, an $\FF_q$-labeled subposet of a poset $\cP$ is just a labeled poset whose poset structure is a subset of $\cP$.

The utility of these definitions comes from the fact that if $\cP \vartriangleleft [[n]]$, then we can naturally identify each supercharacter of $U_\cP$ with an $\FF_q$-labeled poset.  To make this connection explicit, we begin with the following observation.

\begin{proposition}\label{covers-prop}
Fix a poset $\cP\vartriangleleft [[n]]$.   If $\lambda \in \sP_\cP^*$ and $\eta  \in \fkn_\cP^*$ has $\supp(\eta) \subset \coaux_\cP(\lambda)$, then $\supp(\lambda + \eta)$ is the set of covers of a poset $\cP^{(\lambda,\eta)}$ on $[n]$.  
\end{proposition}

\begin{remark}
The superclass analogue of this proposition fails; for example, with $\cP$ as in Example \ref{superclass-example} let $\lambda = 1\larc{a}7 | 2\larc{b}4 \larc{c} 6 | 3| 5 \in \sP_\cP$ and $X = se_{14} + t e_{47} \in \fkn_\cP$.  Then $\supp(X) \subset \aux_\cP(\lambda)$ but $(1,4),(4,7), (1,7) \in \supp(\lambda +X)$, so $\supp(\lambda+X)$ cannot be the set of covers of a poset.
Thus, in somewhat typical asymmetry, the supercharacters of $U_\cP\vartriangleleft  U_n$ appear to lend themselves more naturally to a classification in terms of nice combinatorial objects than do the superclasses.
\end{remark}

\begin{proof}
Fix a set partition $\lambda \in \sP_\cP^*$ and let $S = \supp(\lambda) \cup \coaux_\cP(\lambda)$.  To prove the lemma, it suffices to show that if $(i,j), (j,k) \in S$ then $(i,k) \notin S$, since  if this holds, then $S$ is the set of covers of the poset defined by the condition (\ref{covers2posets}).  We do this by considering four cases.  First, suppose $(i,j), (j,k) \in \supp(\lambda)$.  Since $\lambda$ has at most one nonzero entry in each row and column,  clearly $(i,k) \notin \supp(\lambda)$.  In addition,
by definition $(i,k) \notin \coadj_{[[n]]}(\lambda)$ so $(i,k) \notin \coaux_\cP(\lambda)$.  

Next, suppose $(i,j) \in \supp(\lambda)$ and $(j,k) \in \coaux_\cP(\lambda)$.  Then again $(i,k) \notin\supp(\lambda)$, so, arguing by contradiction, assume $(i,k) \in \coaux_\cP(\lambda)$.   Then by definition $(i,k)$ is below or to the left of some position in $\supp(\lambda)$; the latter case cannot occur since $(i,j) \in \supp(\lambda)$, so there must be some $i' < i$ such that $(i',k) \in \supp(\lambda)$ and $(i',j) \notin \cP$.  But by Lemma \ref{normal-in-U_n} this implies that $(i,j) \notin \cP$, a contradiction.  Hence $(i,k) \notin \coaux_\cP(\lambda)$.  We can handle the symmetric case that $(i,j) \in \coaux_\cP(\lambda)$ and $(j,k) \in \supp(\lambda)$ with a similar argument.

Finally, suppose $(i,j),(j,k) \in \coaux_\cP(\lambda)$.  Then $(i,k) \notin \supp(\lambda)$, since if $(i,k) \in \supp(\lambda)$ and $(i,j ) \in \coaux_\cP(\lambda)$ then we must have $(j,k)\notin \cP$.  Alternatively, if $(i,k) \in \coaux_\cP(\lambda)$ then by definition either there exists $i'<i$ with $(i',k) \in \supp(\lambda)$ and $(i',j) \notin \cP$, or there exists $k'>k$ with $(i,k') \in \supp(\lambda)$ and $(j,k')\notin\cP$.  By Lemma \ref{normal-in-U_n}, however, the first case implies $(i,j) \notin \cP$ and the second case implies $(j,k)\notin \cP$.  Both consequences are contradictions, so necessarily $(i,k)\notin \coaux_\cP(\lambda)$.  
\end{proof}

Extending our notation, we attach to the poset $\cP^{(\lambda,\eta)}$ in Proposition \ref{covers-prop} the obvious $\FF_q$-labeling given by setting
\be \label{labeling}
 \(\cP^{(\lambda,\eta)}\)_{ij} = (\lambda + \eta)_{ij},\qquad\text{for all }i,j.
\ee  When $\cP \vartriangleleft [[n]]$, this gives us a well-defined map
\be\label{partitions2posets} \barr{ccc}  
\Bigl\{
(\lambda, \eta) :  \text{$\lambda \in \sP_\cP^*$ and $\eta \in \fkn_\cP^*$, $\supp(\eta) \subset \coaux_\cP(\lambda)$} 
\Bigr\} & \to & 
\Bigl\{ 
\text{$\FF_q$-labeled posets on $[n]$}
\Bigr\} \\ 
 (\lambda, \eta) & \mapsto & \cP^{(\lambda,\eta)}
\earr\ee
which we will show later to be injective.  

Practically speaking, these observations just mean that we can concisely visualize our supercharacter indexing set by drawing the Hasse diagrams of the corresponding labeled posets.  As we shall soon see, we lose no information by thinking of things in this way.

\begin{example} The Hasse diagram of $\cP^{(\lambda,\eta)}$ has a decomposition into paths (i.e.,  directed graphs whose vertices and edges can be listed as $v_1,\dots,v_m$ and $(v_1,v_2),\dots,(v_{m-1},v_m)$) if and only if $\eta = 0$.  
 In this case, the connected components of the Hasse diagram of $\cP^{(\lambda,0)}$ are just the labeled parts of $\lambda$. 
\end{example}

\begin{example}
Let $\cP \vartriangleleft [[7]]$ be as in Example \ref{superclass-example} and suppose $\lambda =1\larc{a}4\larc{b}6|2\larc{c}7|3\larc{d}5 \in \sP_\cP^*$ and $\eta = r e_{13}^* + s e_{26}^* + t e_{37}^* \in \fkn_\cP^*$ as in Example \ref{supercharacter-example}, where $a,b,c,d \in \FF_q^\times$ and $r,s,t \in \FF_q$.  Then $\cP^{(\lambda,\eta)}$ is the $\FF_q$-labeled poset 
\[\cP^{(\lambda,\eta)}\ =\ \ \xy<0.45cm,1.3cm> \xymatrix@R=.3cm@C=.3cm{
5& &  7   && 6 \\
  & &     & *={ } &   \\
3\ar @{-} [uu]^d  \ar @{-} [uurr]^t &   & 2\ar @{-} [uu]^c \ar @{-} [uurr]^s   && 4 \ar @{-} [uu]_b \\
  &&  &   &  && *={ } \\
  &&   1\ar @{-} [uurr]_a \ar @{-} [lluu]^r &  & *={ } \\
}\endxy\]
where we remove the edges labeled by $r,s,t$, respectively, if these elements are zero.
\end{example}

The problem at the heart of what follows is to determine the image of the map (\ref{partitions2posets}), and then to show that restricted to this set, (\ref{partitions2posets}) has an inverse.  To do this we require some additional definitions.  Fix an arbitrary poset $\cP$ on $[n]$.  We call a subset $S \subset \cP^\cov$ \emph{independent} if no two elements in $S$ share the same first coordinate or same second coordinate; that is, if whenever $(i,j),(k,l) \in S$ are distinct, we have $i\neq k$ and $j\neq l$.  

Among all independent subsets of $\cP^\cov$, there is one of particular interest which we may define as follows.  Given a subset $S\subset \cP^\cov$, let $\ell_S$ denote the integer partition whose parts are the positive numbers $j-i$ for $(i,j) \in S$.  In greater detail, list the elements of $S$ as  $ (i_1,j_1),\dots,(i_s,j_s)$  so that $j_1-i_1 \geq \dots \geq j_s-i_s$ and define $\ell_S$ to be the weakly decreasing sequence of nonnegative integers  
\[ \ell_S = \( j_1-i_1,\ \dots,\  j_s-i_s,\ 0,\ 0,\ 0,\ \dots\).\] 
In particular, if $S = \varnothing$ then $\ell_S  = (0,0,0,\dots)$.  Given any two integer sequences $\ell,\ell'$, we say that  $\ell > \ell'$  if $\ell \neq \ell'$ and the first nonzero coordinate of $\ell-\ell'$ is positive; as usual, we say that $\ell \geq \ell'$ if $\ell = \ell'$ or $\ell > \ell'$.  Under this definition, $\geq$ is the lexicographic total ordering of the set of all integer sequences.

We now define a \emph{highest cover set}  of $\cP$ to be any independent set $S\subset \cP^\cov$ such that $\ell_S\geq \ell_T$ for every independent set $T\subset \cP^\cov$.  Observe that if $S$ is a highest cover set, then $S$ is a maximal independent subset of $\cP^\cov$; i.e., $S$ is not properly contained in any independent set.  This follows simply because any independent set $T$ with $T\supset S$ has $\ell_T \geq \ell_S$.  Consequently, if $S$ is a highest cover set, then every cover has either the same first coordinate or the same second coordinate as some element of $S$.   However, $S$ may not be the independent subset of $\cP^\cov$ with the greatest number of elements.

\def\highest{\mathrm{h}}
\def\lowest{\mathrm{lowest}}

\begin{lemma} If $\cP$ is a poset on $[n]$, then $\cP$ has a unique highest cover set.  
\end{lemma}

We denote the unique highest cover set of $\cP$ by $\cP^\cov_{\highest}$.

\begin{proof}  We induct on the number of elements of $\cP^\cov$.  If $\cP^\cov$ is empty, then $\cP^\cov_{\highest} = \varnothing$ is clearly unique.  Suppose $\cP^\cov$ is not empty, but that if $\cQ$ is any poset on $[n]$ with fewer covers than $\cP$, then $\cQ$ has a unique highest  cover set.
Let $H = \{(i_1,j_1),\dots,(i_s,j_s)\}$ be the set of elements of $\cP^\cov$ with $j_1-i_1 = \dots = j_s-i_s = \max \{ j-i : (i,j) \in \cP^\cov\}$.  Then necessarily the numbers $i_t$ are all distinct and the  numbers $j_t$ are all distinct, so $H$ is independent.  Remove from $\cP$ all elements $(i,j)$ with $i=i_t$ or $j=j_t$ for some $t$, and call the resulting set $\cQ$.  Then $\cQ$ is a poset on $[n]$ with fewer covers than $\cP$, so $\cQ$ has a unique highest cover set $\cQ^\cov_{\highest}$.  The set $S =H \cup \cQ^\cov_{\highest}$ is then an independent subset of $\cP^\cov$.  We claim that it is the unique highest  cover set of $\cP$.  To prove this, suppose $T \subset \cP^\cov$ is an independent subset with $\ell_T \geq \ell_S$.  It suffices to show that $S=T$.  Clearly if $\ell_T \geq \ell_S$ then $H \subset T$, since otherwise one of the first $s$ coordinates of $\ell_T$ would be less than $j_1-i_1 = \dots = j_s-i_s$.   But if $H \subset T$, then $T - H$ is an independent subset of $\cQ^\cov$ with $\ell_{T-H} \geq \ell_{S-H} = \ell_{\cQ^\cov_{\highest}}$, so $T-H = \cQ^\cov_{\highest}$ by hypothesis and $T = S$, as desired. 
\end{proof}

The proof of this lemma shows that to  one can form $\cP^\cov_{\highest}$ by the following algorithm:

\begin{enumerate}
\item[] Input: $\cP$, a poset on $[n]$.
\item[] Instructions: \begin{enumerate}
\item[1.] Set $\tt{S} = \cP^\cov$
\item[2.] Choose an element $(i,k) \in \tt{S}$ such that $\displaystyle k-i = \max_{(x,y) \in \tt{S}} (y-x)$ and add $(i,k)$ to $\cP^\cov_{\highest}$.
\item[3.] Remove from $\tt{S}$ all elements of the form $(i,j)$ and $(j,k)$ for $1\leq j \leq n$.  
\item[4.] If $\tt{S}$ is nonempty, return to Step 2; otherwise, the algorithm terminates.
\end{enumerate} 
\item[] Output: $\cP^\cov_\highest$.
\end{enumerate}

\begin{example}
If $\cP$ is the poset on $n=6$ given by 
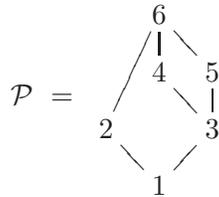
\[ \cP\ =\ \xy<0.25cm,1.1cm> \xymatrix@R=.3cm@C=.3cm{
 & 6 & \\
 & 4 \ar @{-} [u]  & 5 \ar @{-} [ul]  \\
 2\ar @{-} [uur]   &  & 3\ar @{-} [ul]   \ar @{-} [u]    \\
 & 1\ar @{-} [ul]  \ar @{-} [ur]   & 
}\endxy\] then $\cP^\cov_{\highest} = \{ (1,3), (2,6), (3,5)\}$.  Note that $\cP\not\vartriangleleft [[6]]$ since $(2,5)\notin \cP$ but $(3,4) \in \cP$.
\end{example}

By a \emph{chain}, we mean a subset of $[[n]]$ of the form $\{ (a_i, a_j) : 1\leq i <j \leq k\}$ where $1 \leq a_1 < \dots < a_k \leq n$.  
We say that poset $\cP$ \emph{decomposes into chains} if 
there exists a set partition $\lambda = \{\lambda_1,\dots,\lambda_\ell\}$ of $[n]$ such that $\cP$ is the union of the chains $\cC_i \overset{\mathrm{def}} = \{ (a,b) : a,b \in \lambda_i,\ a<b\}$ for $i=1,\dots,\ell$.
This is slightly stronger than saying that $\cP$ is a disjoint union of chains; for example, $\cP = \{ (1,2),(1,3)\}$ is a disjoint union of chains but does not decompose into chains.

\begin{corollary}\label{chain}
If $\cP$ is a poset on $[n]$, then $\cP^\cov = \cP^\cov_\highest$ if and only if $\cP$ decomposes into chains.
\end{corollary}

\begin{proof}
If $\cP$ decomposes into chains then $\cP^\cov$ is independent and clearly the highest cover set.  Conversely, if $\cP^\cov = \cP^\cov_\highest$, then no elements are removed from $\tt{S}$ in step 3 of our algorithm, which means that no two covers of $\cP$ share the same first coordinate or same second coordinate.  Using this property and considering how $\cP$ is determined by its covers, one concludes that $\cP$ decomposes into chains.
\end{proof}

We can now describe a poset condition which characterizes the image of the map (\ref{partitions2posets}).  Given a poset $\cP$ on $[n]$, we say that a poset $\cQ$ on $[n]$ is $\cP$-\emph{representative} if $\cQ \subset \cP$ and the following conditions hold:
\begin{enumerate}
\item[(i)]  If $(i,k) \in \cQ^\cov_{\highest}$ and $(i,j) \in \cQ^\cov$ then $(j,k)\notin \cP$.
\item[(ii)] If $(i,k) \in \cQ^\cov_\highest$ and $(j,k) \in \cQ^\cov$ then $(i,j)\notin \cP$. 
\end{enumerate}
Before proceeding, we inspect what this definition means for the fundamental example $\cP = [[n]]$. 

\begin{proposition}
A poset $\cQ$ is $[[n]]$-representative if and only if $\cQ$ decomposes into chains.  
\end{proposition}

\begin{proof} 
If $\cQ$ is $[[n]]$-representative then no elements can be removed from $\tt{S}$ in step 3 of our algorithm for constructing $\cQ^\cov_\highest$ since this would cause either (i) or (ii) to fail in our definition of $\cP$-representative.  Therefore $\cQ^\cov_\highest = \cQ^\cov$ so $\cQ$ decomposes into chains by Corollary \ref{chain}.  The converse is immediate.
\end{proof}

Each $[[n]]$-representative poset thus naturally corresponds to a set partition of $[n]$: namely, the one whose parts are the connected components of the Hasse diagram of $\cQ$.  
  Hence the set of $[[n]]$-representative $\FF_q$-labeled posets indexes the set of supercharacters of $U_n$.  In fact, we can say something much more general.
  
\begin{theorem}\label{bijections}
Fix a poset $\cP\vartriangleleft [[n]]$.  Then the maps
\[\label{2bijections}
\barr{ccccc}
\Bigl\{
\text{Supercharacters of $U_\cP$}
\Bigr\} 
& 
\to 
& 
\left\{
(\lambda, \eta) : \barr{l} \text{$\lambda \in \sP^*_\cP$ and $\eta \in \fkn_\cP^*$,} \\ \text{$\supp(\eta) \subset \coaux_\cP(\lambda)$} \earr
\right\} 
& 
\to 
&
\left\{ \barr{c} 
\text{$\cP$-representative} \\ \text{$\FF_q$-labeled posets} \earr \right\}
\\ 
\chi_\cP^{(\lambda,\eta)} 
&
\mapsto
&
(\lambda, \eta)
&
\mapsto
&
\cP^{(\lambda,\eta)}
\earr\] are bijections.
\end{theorem}

\begin{remark}
While our definitions make sense in greater generality, the combinatorial picture drawn by this theorem certainly may fail to hold if $U_\cP$ is not normal in $U_n$.  For example, consider the posets  $\cP_{(i)} \subset [[n]]$ of the form
\[  \cP_{(i)}\ \ =\ \ \xy<0.25cm,2.2cm> \xymatrix@R=.3cm@C=.3cm{
& n \\
& \vdots \ar @{-} [u]     \\
& i+1 \ar @{-} [u]     \\ 
1 \ar @{-} [ur] & i \ar @{-} [u]     \\
& \vdots \ar @{-} [u]     \\
 & 2\ar @{-} [u]     
}\endxy\] which Thiem and Venkateswaran study in \cite{TV09}.  For $2<i<n$, we have $\cP_{(i)} \not\vartriangleleft [[n]]$, and in this case the set of $\cP_{(i)}$-representative posets is not in bijection with the set of supercharacters of the corresponding pattern group.
\end{remark}

\begin{proof}
That the first map is a bijection is the statement of Theorem \ref{co-normal-reps}.  To prove that the second map is a bijection, let $\cQ$ be a $\cP$-representative $\FF_q$-labeled poset, and define $\lambda_\cQ,  \eta_\cQ \in \fkn_\cP^*$ by 
\[\label{lambda_Q} \lambda_\cQ = \sum_{(i,j) \in \cQ_\highest^\cov} \cQ_{ij} e_{ij}^*\qquad\text{and}\qquad \eta_\cQ = \sum_{(i,j) \in \cQ^\cov - \cQ_\highest^\cov} \cQ_{ij} e_{ij}^*.\]  We claim that the map 
$ \cQ \mapsto(\lambda_\cQ, \eta_\cQ)$ is the inverse of the second map in the theorem statement.  The difficulty here is just to show that $(\lambda_\cQ,\eta_\cQ)$ lies in the map's domain.  In this direction, we first note that since $\cQ^\cov_\highest$ is an independent subset of $\cQ^\cov$, $\supp(\lambda_\cQ)$ has at most one position in each row and column, so $\lambda_\cQ \in \sP_\cP^*$ is a labeled set partition.  

We want to show that $\supp(\eta_\cQ) \subset \coaux_\cP(\lambda_\cQ)$.  For this, fix some $(j,k) \in \supp(\eta_\cQ) = \cQ^\cov - \cQ^\cov_\highest$.  Since $\supp(\lambda_\cQ) = \cQ^\cov_\highest$, there is either some $(i,k) \in \supp(\lambda_\cQ)$ or some $(j,l) \in \supp(\lambda_\cQ)$, and we want to show that in the first case $i<j$ and in the second case $k<l$.  To this end, suppose the contrary: that any element of $\supp(\lambda_\cQ)$ in the same row or column as $(j,k)$ lies strictly below or to the left of $(j,k)$.  We derive a contradiction as follows.  Let $T_0$ be the set of positions $(j',k') \in \cQ^\cov_\highest$ with $k'-j' \geq k-j$, and observe that $T = T_0 \cup \{(j,k)\}$ remains an independent set by our contrary assumption.  Writing $S = \cQ^\cov_\highest$, we then have the contradiction  $\ell_T > \ell_S$, since if $|T_0| = r$, then by construction $\ell_{T,i} = \ell_{S,i}$ for $i\leq r$ and $\ell_{T,r+1} = k-j > \ell_{S,r+1}$.  

This proves that there is either some $(i,k) \in \supp(\lambda_\cQ)$ with $i<j$ or some $(j,l) \in \supp(\lambda_\cQ)$ with $k<l$.  Hence by (\ref{coadj}) we have $(j,k) \in \coadj_{[[n]]}(\lambda_\cQ)$.  Furthermore, by the definition of $\cP$-representative and (\ref{coadj}), in the first case we have $(i,j)\notin \cP$ so $(j,k)\notin \coadj_{\cP}^{\mathrm L}(\lambda_\cQ)$ and in the second case we have $(k,l)\notin \cP$ so $(j,k)\notin \coadj_{\cP}^{\mathrm R}(\lambda_\cQ)$.  Therefore $(j,k) \notin \coadj_\cP(\lambda_\cQ)$, so $(j,k) \in \coaux_\cP(\lambda_\cQ)$ and more generally $\supp(\eta_\cQ)\subset \coaux_\cP(\lambda_\cQ)$.  

Thus, $(\lambda_\cQ,\eta_\cQ)$ lies is the domain of the second map in the theorem, and clearly its image under this map is $\cP^{(\lambda_\cQ,\eta_\cQ)} = \cQ$.  Hence  $\cQ \mapsto (\lambda_\cQ,\eta_\cQ)$ gives a well-defined inverse to the second map, so this map is   a bijection.
\end{proof}

The preceding proof shows that if $\lambda \in \sP_\cP^*$ and $\eta\in \fkn_\cP^*$ has $\supp(\eta)\subset \coaux_\cP(\lambda)$, then $\supp(\lambda)$ is equal to the highest cover set of $\cP^{(\lambda,\eta)}$.  
Thus, if $\chi_\cQ$ is the supercharacter corresponding to a $\cP$-representative poset $\cQ$, then there is a unique  $\lambda \in \sP_\cP^*$ with $\supp(\lambda) = \cQ^\cov_\highest$ and
\[ \chi_\cQ(1) = q^{|\coadj_\cP^{\mathrm{L}}(\lambda)|} = q^{|\coadj_\cP^{\mathrm{R}}(\lambda)|}\qquad\text{and}\qquad \left\langle \chi_\cQ, \chi_\cQ \right \rangle_{U_\cP} = q^{|\coadj_\cP^{\mathrm L}(\lambda) \cap \coadj_\cP^{\mathrm R}(\lambda)|}.\] In particular,
 we can restate the irreducibility criterion given by Corollary \ref{irred-criterion} as a condition on the $\cP$-representative poset corresponding to a given supercharacter.

\begin{corollary}\label{irred-criterion2}
Let $\cP \vartriangleleft [[n]]$.  Choose a supercharacter $\chi$ of $U_\cP$ and suppose $\chi$ corresponds to the $\cP$-representative $\FF_q$-labeled poset $\cQ$ under the map (\ref{2bijections}).  Then $\chi$ is irreducible if and only if $(i,k),(j,l) \in \cQ^\cov_\highest$ implies $ \{ (i,j) ,(j,k),(k,l)\} \not\subset\cP$ whenever $i<j<k<l$.
\end{corollary}

As an additional corollary, we classify all degree one supercharacters of $U_\cP$ in terms of the corresponding $\cP$-representative posets.

\begin{corollary} Fix a poset $\cP \vartriangleleft[[n]]$.  
\begin{enumerate}
\item[(a)] The poset $\cP$, with any $\FF_q$-labeling, is itself $\cP$-representative and corresponds to a supercharacter of degree one under (\ref{2bijections}).  

\item[(b)] More generally, a supercharacter of $U_\cP$ has degree one if and only if it corresponds under (\ref{2bijections}) to a $\cP$-representative $\FF_q$-labeled poset $\cQ$ satisfying $\cQ^\cov \subset\cP^\cov$. 
\end{enumerate}
\end{corollary}

\begin{proof}
Part (a) follows since, by the definition of a cover, $(i,k),(i,j) \in \cP^\cov \Rightarrow (j,k)\notin \cP$ and $(i,k),(j,k) \in \cP^\cov\Rightarrow (i,j)\notin\cP$.  

To prove part (b), we note that a supercharacter $\chi_\cP^{(\lambda,\eta)}$, where $\lambda \in \sP_\cP^*$ and $\supp(\eta) \subset \coaux_\cP(\lambda)$, has degree one if and only if \[\coadj_\cP^{\mathrm L}(\lambda) = \coadj_\cP^{\mathrm R}(\lambda) = \coadj_\cP(\lambda) = \varnothing.\]  By definition this occurs if and only if $\supp(\lambda) \subset \cP^\cov$.   We claim that in this case $\coaux_\cP(\lambda) \subset \cP^\cov$ as well.  To show this, observe that in this setup \[\coaux_\cP(\lambda) = \coadj_{[[n]]}(\lambda) = \coadj^{\mathrm L}_{[[n]]}(\lambda) \cup \coadj^{\mathrm R}_{[[n]]}(\lambda).\]  If $(j,k) \in \coadj_{[[n]]}^{\mathrm L}(\lambda)$ then there exists $(i,k) \in\supp(\lambda)$ with $i<j$.  If $(j,k)$ is not a cover of $\cP$, so that $(j,j'),(j',k) \in \cP$ for some $j'$, then by Lemma \ref{normal-in-U_n} we have $(i,j'),(j',k) \in \cP$, contradicting $(i,k) \in \cP^\cov$.  If $(j,k) \in \coadj_{[[n]]}^{\mathrm R}(\lambda)$ then it follows that $(j,k) \in \cP^\cov$ by a similar argument.  Hence $\cP^{(\lambda,\eta)} = \supp(\lambda+\eta) \subset \supp(\lambda)\cup \coaux_\cP(\lambda) \subset \cP^\cov$ if and only if $\chi_\cP^{(\lambda,\eta)}$ has degree one, which is the statement of (b).
\end{proof}

We conclude with two examples illustrating applications of Theorem \ref{bijections} to specific pattern groups.

\begin{example}
Suppose $\cP\vartriangleleft [[5]]$ is the poset given by 
 \[\cP\ =\ \xy<0.25cm,0.8cm> \xymatrix@R=.3cm@C=.3cm{
  & 5 \\
 4  & 3   \ar @{-} [u]  \\
2\ar @{-} [u]  \ar @{-} [uur]   & 1 \ar @{-} [ul]  \ar @{-} [u]
}\endxy\qquad\text{so that}\qquad U_\cP = \left\{
\(\barr{ccccc}
1 & 0 & a & b & c \\ & 1 & 0 & d& e \\ & & 1 & 0 & f \\ & & &  1 & 0 \\ & & & & 1 \earr\) : a,b,c,d,e,f \in \FF_q \right\}\]
in which case $U_\cP$ is the commutator subgroup of $U_{5}$.  Then the $\cP$-representative $\FF_q$-labeled posets are those of the following fifteen forms
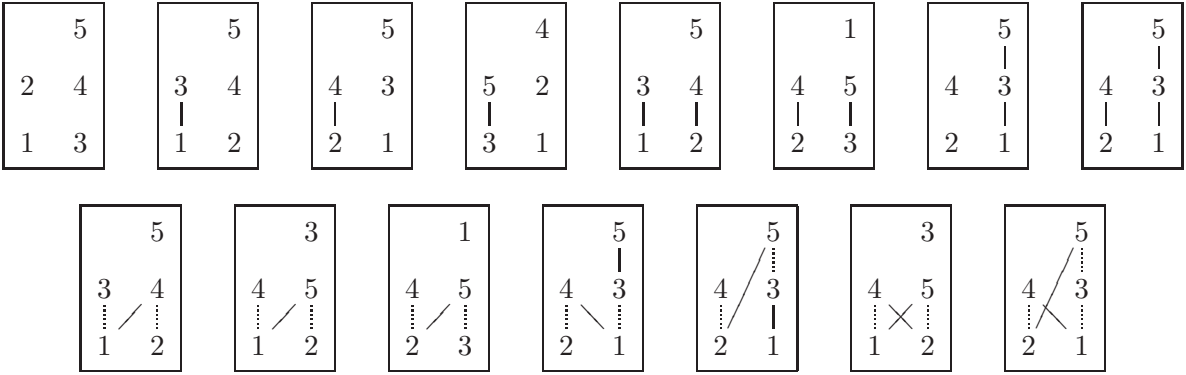
\[\barr{c} 
\barr{ccccccccccccccc} 
\boxed{\xy<0.25cm,0.8cm> \xymatrix@R=.3cm@C=.3cm{
  & 5 \\
 2  & 4 \\
1  & 3
}\endxy} 
&
&\boxed{\xy<0.25cm,0.8cm> \xymatrix@R=.3cm@C=.3cm{
  & 5 \\
 3  & 4\\
1  \ar @{-} [u] & 2
}\endxy} 
&
& \boxed{\xy<0.25cm,0.8cm> \xymatrix@R=.3cm@C=.3cm{
  & 5 \\
 4  & 3\\
2  \ar @{-} [u] & 1
}\endxy} 
&
&\boxed{\xy<0.25cm,0.8cm> \xymatrix@R=.3cm@C=.3cm{
  & 4 \\
 5  & 2\\
3  \ar @{-} [u] & 1
}\endxy} 
&
&\boxed{\xy<0.25cm,0.8cm> \xymatrix@R=.3cm@C=.3cm{
  & 5 \\
 3  & 4\\
1  \ar @{-} [u] & 2  \ar @{-} [u]
}\endxy} 
&
&\boxed{\xy<0.25cm,0.8cm> \xymatrix@R=.3cm@C=.3cm{
  & 1 \\
 4  & 5\\
2  \ar @{-} [u] & 3  \ar @{-} [u]
}\endxy} 
&
&\boxed{\xy<0.25cm,0.8cm> \xymatrix@R=.3cm@C=.3cm{
  & 5 \\
 4  & 3\ar @{-} [u] \\
2   & 1  \ar @{-} [u]
}\endxy} 
&
&\boxed{\xy<0.25cm,0.8cm> \xymatrix@R=.3cm@C=.3cm{
  & 5 \\
 4  & 3\ar @{-} [u] \\
2\ar @{-} [u]   & 1  \ar @{-} [u]
}\endxy} \earr
\\ \\
\barr{ccccccccccccc}
\boxed{\xy<0.25cm,0.8cm> \xymatrix@R=.3cm@C=.3cm{
  & 5 \\
 3  & 4  \\
1\ar @{.} [u]  \ar @{-} [ur]   & 2  \ar @{.} [u]
}\endxy} 
&
&
\boxed{\xy<0.25cm,0.8cm> \xymatrix@R=.3cm@C=.3cm{
  & 3 \\
 4  & 5  \\
1\ar @{.} [u]  \ar @{-} [ur]   & 2  \ar @{.} [u]
}\endxy} 
&
&
\boxed{\xy<0.25cm,0.8cm> \xymatrix@R=.3cm@C=.3cm{
  & 1 \\
 4  & 5  \\
2\ar @{.} [u]  \ar @{-} [ur]   & 3  \ar @{.} [u]
}\endxy} 
&
&
\boxed{\xy<0.25cm,0.8cm> \xymatrix@R=.3cm@C=.3cm{
 & 5   \\
 4 & 3\ar @{-} [u]     \\
2  \ar @{.} [u] & 1\ar @{.} [u]  \ar @{-} [ul]   
}\endxy} 
&
&
\boxed{\xy<0.25cm,0.8cm> \xymatrix@R=.3cm@C=.3cm{
  &  5 \\
 4&  3\ar @{.} [u]    \\
2   \ar @{.} [u]  \ar @{-} [uur] & 1\ar @{-} [u]     
}\endxy} 
&
&
\boxed{\xy<0.25cm,0.8cm> \xymatrix@R=.3cm@C=.3cm{
  & 3 \\
 4  & 5  \\
1\ar @{.} [u]  \ar @{-} [ur]   & 2 \ar @{-} [ul]  \ar @{.} [u]
}\endxy} 
&
&
\boxed{\xy<0.25cm,0.8cm> \xymatrix@R=.3cm@C=.3cm{
  & 5 \\
 4  & 3   \ar @{.} [u]  \\
2\ar @{.} [u]  \ar @{-} [uur]   & 1 \ar @{-} [ul]  \ar @{.} [u]
}\endxy} \earr
\earr\]
where both solid and dashed lines are labeled by elements of $\FF_q^\times$, but dashed lines are optional and may be omitted.  (Alternatively, we can think of dashed lines as being labeled by elements of $\FF_q$ rather than $\FF_q^\times$.)  In each of these diagrams, the solid lines indicate the elements of the given poset's highest cover set.
Counting the number of such labelings, it follows by Theorem \ref{bijections} that the number of supercharacters and superclasses of $U_\cP$ is 
\[ 1 + 3(q-1) + 3(q-1)^2 + (q-1)^3 + 3 q^2 (q-1) + 3q^2(q-1)^2 + q^3(q-1)^2\] which simplifies to $q^3(q^2 + 1)(q-1)$.  All of these representative posets satisfy the condition in Corollary \ref{irred-criterion2}, so every supercharacter of $U_\cP$ is irreducible and $q^3(q^2 + 1)(q-1)$ is the number of irreducible characters and conjugacy classes of the group.
\end{example}

As our second example, we index the supercharacters of a family of pattern groups.

\begin{example}
Given nonnegative integers $m,n$, let $\cT = \cT(m,n) \vartriangleleft[[m+n]]$ denote the poset given by 
 \[\cT(m,n)\ =\ \xy<0.25cm,1.2cm> \xymatrix@R=.3cm@C=.3cm{
  m+1 & m+2 & \cdots & m+n \\
  m  \ar @{-} [u]  \ar @{-} [ur] \ar @{-} [urrr] &  &  &  \\
 \vdots \ar @{-} [u]    \\
1\ar @{-} [u]   
}\endxy
\quad\text{so that}\quad U_\cT = \left\{
\(\barr{cc}
u & a \\ 0 & I_n\earr\) : u \in U_m,\ a \in \FF_q^{m\times n} \right\}.
\] This pattern group is isomorphic to the semidirect product of $U_m$ and the additive group $\FF_q^{m\times n}$ of $m\times n$ matrices over $\FF_q$, with respect to the natural action of $U_m$ on $\FF_q^{m\times n}$ by left multiplication.  

Using Theorems \ref{co-normal-reps} and \ref{bijections}, it follows that a subposet $\cP$ is $\cT$-representative if and only if $\cP$ is of the form 
\[\cP \ = \ {\xy<0.25cm,1.6cm> \xymatrix@R=.3cm@C=.3cm{
  \ell_1^{(1)} & \ell_2^{(1)} & \cdots & \ell_{k_1}^{(1)} \\
  v_{j_1}^{(1)}  \ar @{-} [u]  \ar @{-} [ur] \ar @{-} [urrr] &  &  &  \\
 \vdots \ar @{-} [u]    \\
 v_1^{(1)}\ar @{-} [u]   
}\endxy\quad 
\xy<0.25cm,1.6cm> \xymatrix@R=.3cm@C=.3cm{
  \ell_1^{(2)} & \ell_2^{(2)} & \cdots & \ell_{k_2}^{(2)} \\
  v_{j_2}^{(2)}  \ar @{-} [u]  \ar @{-} [ur] \ar @{-} [urrr] &  &  &  \\
 \vdots \ar @{-} [u]    \\
 v_1^{(2)}\ar @{-} [u]   
}\endxy
\quad \cdots\quad \quad 
\xy<0.25cm,1.6cm> \xymatrix@R=.3cm@C=.3cm{
  \ell_1^{(r)} & \ell_2^{(r)} & \cdots & \ell_{k_r}^{(r)} \\
  v_{j_r}^{(r)}  \ar @{-} [u]  \ar @{-} [ur] \ar @{-} [urrr] &  &  &  \\
 \vdots \ar @{-} [u]    \\
 v_1^{(r)}\ar @{-} [u]   
}\endxy}
\] for some nonnegative integers $r$, $j_1,\dots,j_r$, $k_1,\dots,k_r$, where the following conditions hold:
\begin{enumerate}
\item[(a)] $v_j^{(i)} \in \{1,\dots, m\}$ and $\ell_{j}^{(i)} \in \{m+1,\dots,m+n\}$ for each $i,j$.
\item[(b)] The vertices $v_j^{(i)}$ are distinct for all $i,j$.
\item[(c)]  $\ell_{j}^{(i)} \neq \max\left\{ \ell_{1}^{(i')}, \dots, \ell_{k_{i'}}^{(i')} \right\}$ if $i\neq i'$ and $v_{j_i}^{(i)} > v_{j_{i'}}^{(i')}$.
\end{enumerate}  Notably, the vertices $\ell_j^{(i)}$ in the Hasse diagram of $\cP$ may coincide, subject to condition (c).  For example, if $m=5$ and $n=3$ and we have two posets   
\[ \cP_1 \ =  \ \xy<0.25cm,.75cm> \xymatrix@R=.3cm@C=.3cm{
  8& 7 & 6 \\
  3  \ar @{-} [u]  \ar @{-} [ur]  &4\ar @{-} [u] \ar @{-} [ur] &    \\
 1\ar @{-} [u] & 2 \ar @{-} [u]  & 5
}\endxy\qquad \text{and}\qquad
\cP_2 \ =  \ \xy<0.25cm,.75cm> \xymatrix@R=.3cm@C=.3cm{
  8& 7 & 6 \\
  4  \ar @{-} [u]  \ar @{-} [ur]  &3\ar @{-} [u] \ar @{-} [ur] &    \\
 1\ar @{-} [u] & 2 \ar @{-} [u]  & 5
}\endxy\]  then $\cP_1$ is $\cT$-representative while $\cP_2$ is not $\cT$-representative. 
\end{example}

\end{document}